\documentclass{article}

\usepackage{authblk}
\usepackage{a4wide}
\usepackage{amsthm}
\usepackage{graphicx}
\usepackage{amssymb}
\usepackage{amsmath}
\usepackage{ascmac}
\usepackage{setspace}
\usepackage{float}
\usepackage{algorithm}
\usepackage{algorithmic}
\usepackage{setspace}
\usepackage{tikz-cd}
\usepackage{comment}
\usepackage{appendix}
\newtheorem{Theorem}[equation]{Theorem}
\newtheorem{Corollary}[equation]{Corollary}
\newtheorem{Lemma}[equation]{Lemma}

\theoremstyle{definition}
\newtheorem{Definition}[equation]{Definition}
\newtheorem{Conjecture}[equation]{Conjecture}
\theoremstyle{remark}

\numberwithin{equation}{section}

\DeclareMathOperator{\ad}{ad}

\allowdisplaybreaks
\title{New presentation of the twisted Yangian of type $D$}
\author{Shuichi Harako\thanks{Graduate school of Mathematical Sciences, the University of Tokyo, harako@ms.u-tokyo.ac.jp} and Mamoru Ueda\thanks{Graduate school of Mathematical Sciences, the University of Tokyo, mueda@ms.u-tokyo.ac.jp}}
\date{}

\begin{document}

\maketitle

\begin{abstract}
We construct a new presentation of the twisted Yangian of type $D$ with a finite number of generators. By using this presentation, we propose a new definition of the twisted affine Yangian of type $D$.
\end{abstract}
\section{Introduction}
Drinfeld (\cite{D1}, \cite{D2}) introduced the Yangian $Y_h(\mathfrak{g})$ associated with a finite dimensional simple Lie algebra $\mathfrak{g}$. The Yangian $Y_h(\mathfrak{g})$ is a quantum group which is a deformation of a current algebra $\mathfrak{g}\otimes\mathbb{C}[z]$. The Yangian has several presentations including the RTT presentation, the Drinfeld $J$ presentation, the Drinfeld presentation. By using the Drinfeld presentation, the definition of the Yangian can be extended to a symmetrizable Kac-Moody Lie algebra. For an affine Lie algebra, Guay-Nakajima-Wendlandt \cite{GNW} constructed a finite presentation of the Yangian and showed that the Yangian associated with an affine Lie algebra has a coproduct. Recently, relationships between Yangians and $W$-algebras have been actively studied. In finite setting, there exists a relationship between the Yangian of type $A$ and a finite $W$-algebra of type $A$ (see Ragoucy-Sorba \cite{RS} and Brundan-Kleshchev \cite{BK}). As for other types, a finite $W$-algebra is connected to the twisted Yangian (see \cite{Bro}, \cite{DKV} and \cite{LPTTW}). 

Olshanskii (\cite{O}) introduced the twisted Yangian associated with an orthogonal Lie algebra $\mathfrak{so}(n)$ or a symplectic Lie algebra $\mathfrak{sp}(\dfrac{n}{2})$. Similarly to the finite Yangian $Y_\hbar(\mathfrak{sl}(n))$, the twisted Yangian also has several presentations. Olshanskii defined the twisted Yangian by the RTT presentation. Belliard and Regelskis (\cite{BR1}) constructed the Drinfeld $J$ presentation of the twisted Yangian. Lu-Wang-Zhang \cite{LWZ0}, \cite{LWZ} and Lu-Zhang \cite{LWZ1} gave the Drinfeld presentation of the twisted Yangian.
In this article, we give a new presentation of the twisted Yangian of type $D$, which corresponds to the presentation given in \cite{GNW}. 

In this article, we only consider the twisted Yangian associated with $\mathfrak{so}(2n)$. A Lie algebra $\mathfrak{so}(2n)$ can be regarded as a fixed point subalgebra in $\mathfrak{sl}(2n)$ of an appropriate involution of $\mathfrak{sl}(2n)$. We can choose this involution by two ways. The first one is the transpose of an $2n\times 2n$-matrix. The another one is given by
\begin{align*}
E_{i,j}\mapsto -E_{j+n,i+n},\ E_{i,j+n}\mapsto -E_{j,i+n},\ E_{i+n,j}\mapsto -E_{j+n,i}\text{ for }1\leq i,j\leq n,
\end{align*}
where $E_{i,j}$ is a matrix unit whose $(i,j)$-component is $1$ and other components are $0$. Lu-Wang-Zhang \cite{LWZ0} gave the Drinfeld presentation by using the first involution. In this article, we use the latter involution and denote it by $\tau$. 

In Section 3, we give a presentation of the universal enveloping algebra of a twisted current Lie algebra $\mathfrak{sl}(2n)[z]^{\widetilde{\tau}}=\{x\in\mathfrak{sl}(2n)[u]\mid\widetilde{\tau}(x)=x\}$, where $\widetilde{\tau}$ is an involution of $\mathfrak{sl}(2n)[z]$ defined by $\widetilde{\tau}(xz^r)=(-1)^r\tau(x)z^r$. The set of generators of this presentation is
\begin{align*}
\{h_iz^{2r},x^\pm_iz^{2r},\bar{h}_jz^{2r+1},\bar{x}^\pm_kz^{2r+1}\mid 1\leq i\leq n,1\leq j\leq n-1,1\leq k\leq n+1,r\geq0\},
\end{align*}
where
\begin{gather*}
h_i=\begin{cases}
E_{i,i}-E_{i+1,i+1}-E_{n+i,n+i}+E_{n+i+1,n+i+1}&\text{ if }1\leq i\leq n-1,\\
E_{n-1,n-1}+E_{n,n}-E_{2n-1,2n-1}-E_{2n,2n}&\text{ if }i=n,
\end{cases}\\
x^+_i=\begin{cases}
E_{i,i+1}-E_{n+i+1,n+i}&\text{ if }1\leq i\leq n-1,\\
E_{n-1,2n}-E_{n,2n-1}&\text{ if }i=n,
\end{cases}\ 
x^-_i=\begin{cases}
E_{i+1,i}-E_{n+i,n+i+1}&\text{ if }1\leq i\leq n-1,\\
E_{2n,n-1}-E_{2n-1,n}&\text{ if }i=n,
\end{cases}\\
\bar{h}_i=
E_{i,i}-E_{i+1,i+1}+E_{n+i,n+i}-E_{n+i+1,n+i+1}\text{ if }1\leq i\leq n-1,\\
\bar{x}^+_i=
\begin{cases}
E_{i,i+1}+E_{n+i+1,n+i}&\text{ if }1\leq i\leq n-1,\\
E_{n-1,2n}+E_{n,2n-1}&\text{ if }i=n,\\
E_{n,2n}&\text{ if }i=n+1,
\end{cases}\ 
\bar{x}^-_i=
\begin{cases}
E_{i+1,i}+E_{n+i,n+i+1}&\text{ if }1\leq i\leq n-1,\\
E_{2n,n-1}+E_{2n-1,n}&\text{ if }i=n,\\
E_{2n,n}&\text{ if }i=n+1.
\end{cases}
\end{gather*}

In Section 4, we construct a presentation of the twisted current Lie algebra $U(\mathfrak{sl}(2n)[z]^{\widetilde{\tau}})$ with the generators $\{h_i,x^\pm_i,\bar{h}_jz,\bar{x}^\pm_iz\mid 1\leq i\leq n,1\leq j\leq n-1\}$. In Section 6, by using the presentation given in Section~4, we define an associative algebra $TY_\hbar(\mathfrak{so}(2n))$ as a deformation of $U(\mathfrak{sl}(2n)[z]^{\widetilde{\tau}})$. 

\begin{Theorem}
The associative algebra $TY_\hbar(\mathfrak{so}(2n))$ is isomorphic to the twisted Yangian associated with $\mathfrak{so}(2n)$.
\end{Theorem}
Since Guay \cite{Gu1} defined the affine Yangian associated with $\widehat{\mathfrak{sl}}(n)$ by extending the minimalistic presentation of the Yangian associated with $\mathfrak{sl}(n)$, we expect that this finite presentation will be useful for giving a definition of the twisted affine Yangian. In \cite{U77}, the author defined the twisted affine Yangian of type $C$ as a subalgebra of the affine Yangian associated with $\widehat{\mathfrak{sl}}(n)$ and constructed a relationship between the twisted affine Yangian of type $C$ and a rectangular $W$-algebra of type $D$ in certain cases. However, this definition is not so useful because the defining relations of the twisted affine Yangian were not given. In Section 7, we propose a new definition of the twisted affine Yangian of type $D$ by extending the definition of $TY_\hbar(\mathfrak{so}(2n))$. We conjecture that this twisted affine Yangian will become a coideal of the affine Yangian associated with $\widehat{\mathfrak{sl}}(2n)$. 

\section{Lie algebra $\mathfrak{so}(2n)$}

Let $n$ be a positive integer greater than $4$ and define a set $I_n=\{\pm1,\pm2,\cdots,\pm n\}$. 
We set an involution of the general linear Lie algebra $\mathfrak{gl}(2n)=\bigoplus_{i,j\in I_n}\mathbb{C}E_{i,j}$ by $\tau(E_{i,j})=-E_{-j,-i}$,
where $E_{i,j}$ is a $2n\times2n$-matrix whose $(p,q)$ component is $\delta_{i,p}\delta_{j,q}$. By the definition of $\tau$, the restriction of $\tau$ to $\mathfrak{sl}(2n)$ becomes an involution of $\mathfrak{sl}(2n)$. Let us set two subalgebras of $\mathfrak{sl}(2n)$:
\begin{align*}
\mathfrak{b}_1&=\{x\in\mathfrak{sl}(2n)\mid \tau(x)=x\}\\
&=\{\begin{pmatrix}
A&B \\
C&-A^t 
\end{pmatrix}\in\mathfrak{gl}(2n)\mid A,B,C\in\mathfrak{gl}(n),B^t=-B,C^t=-C\},\\
\mathfrak{b}_2&=\{x\in\mathfrak{sl}(n)\mid \tau(x)=-x\}\\
&=\{\begin{pmatrix}
A&B \\
C&A^t 
\end{pmatrix}\in\mathfrak{gl}(2n)\mid A\in\mathfrak{sl}(n),B,C\in\mathfrak{gl}(n),B^t=B,C^t=C\},
\end{align*}
where $X^t$ is the transpose of $X$ for an $n\times n$-matrix $X$.
It is well known that $\mathfrak{b}_1$ is isomorphic to the orthogonal Lie algebra $\mathfrak{so}(2n)$.

Let us set $f_{i,j}$ as $E_{i,j}-E_{-j,-i}$ for $i,j\in I_n$.
By the definition, $\mathfrak{b}_1$ has a basis
\begin{equation*}
\{f_{i,j}\mid i,j>0\}\cup\{f_{i,j}\mid ij<0,i\neq-j,|i|<|j|\}
\end{equation*}
and has a non-degenerate symmetric invariant bilinear form $(\ ,\ )$ defined by $(f_{i,j},f_{p,q})=\delta_{i,q}\delta_{j,p}-\delta_{i,-p}\delta_{j,-q}$.
We set $(a_{i,j})_{1\leq i,j\leq n}$ as the Cartan matrix of $\mathfrak{so}(2n)$:
\begin{align*}
a_{i,j}&=\begin{cases}
2\delta_{i,j}-\delta_{i,j+1}-\delta_{i+1,j}&\text{ if }1\leq i,j\leq n-1,\\
2\delta_{i,j}-\delta_{i,n-2}-\delta_{j,n-2}&\text{ if }i\text{ or }j=n
\end{cases}
\end{align*}
We also take the Chevalley generators of $\mathfrak{b}_1$ as follows:
\begin{gather*}
h_i=\begin{cases}
f_{i,i}-f_{i+1,i+1}&\text{ if }1\leq i\leq n-1,\\
f_{n-1,n-1}+f_{n,n}&\text{ if }i=n,
\end{cases}\\
x^+_i=\begin{cases}
f_{i,i+1}&\text{ if }1\leq i\leq n-1,\\
f_{n-1,-n}&\text{ if }i=n,
\end{cases}\ 
x^-_i=\begin{cases}
f_{i+1,i}&\text{ if }1\leq i\leq n-1,\\
f_{-n,n-1}&\text{ if }i=n.
\end{cases}
\end{gather*}
\section{Twisted current algebra $\mathfrak{sl}(2n)[u]^{\widetilde{\tau}}$}
The current algebra $\mathfrak{sl}(n)[u]$ is a Lie algebra whose vector space is $\mathfrak{sl}(n)\otimes\mathbb{C}[u]$ with the commutator relation:
\begin{equation*}
[x\otimes u^r,y\otimes u^s]=[x,y]\otimes u^{r+s}\text{ for }x,y\in\mathfrak{sl}(n).
\end{equation*}
Let $(A_{i,j})_{1\leq i,j\leq n}$ be the Cartan matrix of $\mathfrak{sl}(n)$. It is famous that the current algebra $\mathfrak{sl}(n)[u]$ has the following presentation.
\begin{Lemma}
The associative algebra $U(\mathfrak{sl}(n)[u])$ is isomorphic to the associative algebra generated by
\begin{equation*}
\{\mathfrak{h}_{i,r},\mathfrak{x}^\pm_{i,r}\mid1\leq i\leq n-1,r\in\mathbb{Z}_{\geq0}\}
\end{equation*}
with the relations:
\begin{gather}
[\mathfrak{h}_{i,r},\mathfrak{h}_{j,s}]=0,\label{last1}\\
[\mathfrak{h}_{i,0},\mathfrak{x}^\pm_{j,s}]=A_{i,j}\mathfrak{x}^\pm_{j,s},\label{last2}\\
[\mathfrak{x}^+_{i,r},\mathfrak{x}^-_{j,s}]=\delta_{i,j}\mathfrak{h}_{i,r+s},\label{last3}\\
[\mathfrak{x}^\pm_{i,r},\mathfrak{x}^\pm_{j,s}]=0\text{ for }A_{i,j}=0,\label{last4}\\
[\mathfrak{x}^\pm_{i,r_1},[\mathfrak{x}^\pm_{i,r_2},\mathfrak{x}^\pm_{j,s}]]=0\text{ for }A_{i,j}=-1,\label{last5}\\
[\mathfrak{h}_{i,r+1},\mathfrak{x}^\pm_{j,s}]=[\mathfrak{h}_{i,r},\mathfrak{x}^\pm_{j,s+1}],\label{last6}\\
[\mathfrak{x}^\pm_{i,r+1},\mathfrak{x}^\pm_{j,s}]=[\mathfrak{x}^\pm_{i,r},\mathfrak{x}^\pm_{j,s+1}].\label{last7}
\end{gather}
Moreover, the generators $\{\mathfrak{x}^+_{i,r}\mid 1\leq i\leq n-1,r\in\mathbb{Z}_{\geq0}\}$ and the relations \eqref{last4}, \eqref{last5} and \eqref{last7} give a Lie subalgebra $\mathfrak{n}_+(n)=\bigoplus_{1\leq i<j\leq n}\limits\mathbb{C}E_{i,j}u^r\subset\mathfrak{sl}(n)[u]$.
\end{Lemma}

We define an involution of the current algebra $\mathfrak{sl}(2n)[u]$ by $\widetilde{\tau}(x\otimes u^r)=(-1)^r\tau(x)\otimes u^r$.
By the definition of $\widetilde{\tau}$, the twisted current algebra $\mathfrak{sl}(2n)[u]^{\widetilde{\tau}}=\{x\in\mathfrak{sl}(2n)[u]\mid \widetilde{\tau}(x)=x\}$ can be decomposed to the sum of $\{x\otimes u^{2r}\mid x\in\mathfrak{b}_1\}$ and $\{x\otimes u^{2r+1}\mid x\in\mathfrak{b}_2\}$. Let us set $f^r_{i,j}=E_{i,j}u^r-(-1)^rE_{-j,-i}u^r$, which satisfies the relation
\begin{equation*}
[f^r_{i,j},f^s_{p,q}]=\delta_{j,p}f^{r+s}_{i,q}-\delta_{i,q}f^{r+s}_{p,j}-(-1)^r\delta_{-i,p}f^{r+s}_{-j,q}+(-1)^r\delta_{-j,q}f^{r+s}_{p,-i}.
\end{equation*}
Then, by the definition, $f^r_{i,j}$ spans the twisted current algebra $\mathfrak{sl}(2n)[u]^{\widetilde{\tau}}$.  We also take generators of $\mathfrak{sl}(2n)[u]^\tau$ as
\begin{gather*}
h_{i,r}=\begin{cases}
f^r_{i,i}-f^r_{i+1,i+1}&\text{ if }1\leq i\leq n-1,\\
f^r_{n-1,n-1}+f^r_{n,n}&\text{ if }i=n,r\text{ is }even,
\end{cases}\\
x^+_{i,r}=\begin{cases}
f^r_{i,i+1}&\text{ if }1\leq i\leq n-1,\\
f^r_{n-1,-n}&\text{ if }i=n,\\
f^r_{n,-n}&\text{ if }i=n+1,r\text{ is }odd,
\end{cases}\ 
x^-_{i,r}=\begin{cases}
f^r_{i+1,i}&\text{ if }1\leq i\leq n-1,\\
f^r_{-n,n-1}&\text{ if }i=n,\\
f^r_{-n,n}&\text{ if }i=n+1,r\text{ is }odd.
\end{cases}
\end{gather*}

For a non-negative integer $r$, we set matrices $(a^r_{i,j})$ as
\begin{align*}
a^r_{i,j}=\begin{cases}
a_{i,j}&\text{ if }1\leq i,j\leq n,r\text{ is even,}\\
a_{i,j}+2\delta_{i,n-1}\delta_{j,n}&\text{ if }1\leq i\leq n-1,1\leq j\leq n,r\text{ is odd}.
\end{cases}
\end{align*}
Then, we have $[h_{i,r},x^\pm_{j,s}]=\pm a^r_{i,j}x^\pm_{j,r+s}=0$ for $j\neq n+1$. 

In this section, we will give a presentation of the universal enveloping algebra of $\mathfrak{sl}(2n)[u]^\tau$ whose generators are $\{h_{i,r},x^\pm_{i,r}\}$.
\begin{Definition}
We define $L$ as the associative algebra generated by
\begin{equation*}
\{H_{i,2r},X^\pm_{i,2r}\mid 1\leq i\leq n,r\in\mathbb{Z}_{\geq0}\}\cup\{H_{i,2r+1},X^\pm_{j,2r+1}\mid 1\leq i\leq n-1,1\leq j\leq n+1,r\in\mathbb{Z}_{\geq0}\}
\end{equation*}
with the commutator relations:
\begin{gather}
[H_{i,r},H_{j,s}]=0,\label{Eq901}\\
[H_{i,r},X^\pm_{j,s}]=\pm a^{r}_{i,j}X^\pm_{j,r+s}\text{ for }j\neq n+1,\label{Eq901-1}\\
[H_{i,r},X^\pm_{n+1,s}]=\begin{cases}
0&r\text{ is odd},\\
\pm(2\delta_{i,n}-2\delta_{i,n-1})X^\pm_{n+1,r+s}&r\text{ is even},
\end{cases}\label{Eq901-1-1}\\
[X^+_{i,r},X^-_{j,s}]=\delta_{i,j}H_{i,r+s}\text{ for }1\leq i,j\leq n-1,\label{Eq901-2}\\
[X^+_{n,r},X^-_{j,s}]=\begin{cases}
\delta_{n,j}H_{n,r+s}&\text{ if }j\neq n-1,n+1,r+s\text{ is even},\\
\delta_{n,j}H_{n-1,r+s}&\text{ if }j\neq n-1,n+1,r+s\text{ is odd},
\end{cases}\label{Eq901-3}\\
[X^+_{n,r},X^-_{n-1,s}]=0\text{ if }r+s\text{ is even},\label{Eq901-4}\\
[X^+_{n,r},X^-_{n-1,s}]=-X^+_{n+1,r+s}\text{ if }r+s\text{ is odd},\label{Eq901-5}\\
[X^+_{j,r},X^-_{n,s}]=\delta_{n,j}H_{n,r+s}\text{ if }j\neq n-1,n+1,\label{Eq901-6}\\
[X^+_{n-1,r},X^-_{n,s}]=0\text{ if }r+s\text{ is even},\label{Eq901-7}\\
[X^+_{n-1,r},X^-_{n,s}]=-X^-_{n+1,r+s}\text{ if }r+s\text{ is odd},\label{Eq901-8}\\
[X^+_{n+1,2r+1},X^-_{n+1,2s+1}]=2(H_{n,2r+2s+2}-H_{n-1,2r+2s+2}),\label{Eq901-9}\\
[X^\pm_{i,r+1},X^\pm_{j,s}]=[X^\pm_{i,r},X^\pm_{j,s+1}]\text{ for }1\leq i,j\leq n\text{ and }(i,j)\neq(n-1,n),(n,n-1),\label{Eq901-10-0}\\
[X^\pm_{n,r+1},X^\pm_{n-1,s}]=-[X^\pm_{n,r},X^\pm_{n-1,s+1}],\label{Eq901-10}\\
[X^\pm_{n+1,2r+1},X^\pm_{n+1,2s+1}]=0,\label{Eq901-10-1}\\
[X^\pm_{i,r},X^\pm_{i,s}]=0\text{ for }1\leq i\leq n,\label{Eq901-10-2}\\
[X^+_{n+1,2r+1},X^\pm_{j,s}]=0\text{ for }j<n-1,\label{Eq901-11}\\
[X^-_{n+1,2r+1},X^\pm_{j,s}]=0\text{ if }j<n-1,\label{Eq901-12}\\
[X^+_{n+1,2r+1},X^+_{n-1,s}]=-2X^+_{n,2r+s+1},\ [X^+_{n+1,2r+1},X^-_{n-1,s}]=0,\label{Eq901-13}\\
[X^+_{n+1,2r+1},X^-_{n,s}]=2X^-_{n-1,2r+s+1},\ [X^+_{n+1,2r+1},X^+_{n,s}]=0,\label{Eq901-14}\\
[X^-_{n+1,2r+1},X^-_{n-1,s}]=2X^-_{n,2r+s+1},\ [X^-_{n+1,2r+1},X^+_{n-1,s}]=0,\label{Eq901-15}\\
[X^-_{n+1,2r+1},X^+_{n,s}]=-2X^+_{n-1,2r+s+1},\ [X^-_{n+1,2r+1},X^-_{n,s}]=0,\label{Eq901-16}\\
[X^\pm_{i,r},X^\pm_{j,s}]=0\text{ if }a_{i,j}=0\text{ and }(i,j)\neq (n-1,n),(n,n-1),\label{Eq901-17}\\
[X^\pm_{n-1,r},X^\pm_{n,s}]=0\text{ if }r+s\text{ is even},\label{Eq901-17-1}\\
[X^\pm_{i,r_1},[X^\pm_{i,r_2},X^\pm_{j,s}]]=0\text{ if }(i,j)=(n-1,n),(n,n-1)\text{ and }r_1+s,r_2+s\text{ are odd},\label{Eq901-17-2}\\
[X^\pm_{i,r_1},[X^\pm_{i,r_2},X^\pm_{j,s}]]=0\text{ if }a_{i,j}=-1\text{ and }1\leq i,j\leq n,\label{Eq902}\\
[X^\pm_{n,s},[X^\pm_{n-2,r},X^\pm_{n-1,u}]]=(-1)^{s+u}[X^\pm_{n-1,u},[X^\pm_{n-2,r},X^\pm_{n,s}]],\label{Eq903-1}\\
[[X^\pm_{n,0},X^\pm_{n-2,r}],[X^\pm_{n,0},X^\pm_{n-1,s}]]=0,\label{Eq903-3}\\
[[X^+_{n,s},[X^+_{n-2,0},X^+_{n-3,0}]],[[X^+_{n-1,0},X^+_{n-2,0}],X^+_{n,r}]]=0,\label{Eq903-4}\\
[[X^+_{n,2r+1},X^+_{n-2,0}],[[X^+_{n-1,0},X^+_{n-2,0}],X^+_{n,s}]]=0,\label{Eq903-5},\\
[[X^\pm_{n-1,0},X^\pm_{n-2,0}],[X^\pm_{n-1,s},X^\pm_{n,r}]]=0,\label{Eq903-7}\\
[[X^+_{n-1,s},[X^+_{n-2,0},X^+_{n-3,0}]],[[X^+_{n-1,0},X^+_{n-2,0}],X^+_{n,r}]]=0,\label{Eq903-8}\\
[[X^+_{n-1,s},X^+_{n-2,0}],[[X^+_{n-1,0},X^+_{n-2,0}],X^+_{n,r}]] = 0.\label{Eq903-8.5}
\end{gather}
\end{Definition}
In this section, we will show that $L$ is isomorphic to $U(\mathfrak{sl}(2n)^{\widetilde{\tau}})$.

For positive integers $1\leq i<j\leq n$, we set
\begin{align}
X_{i,j,r}&=\prod_{u=i}^{j-2}\ad(X^+_{u,0})(X^+_{j-1,r}),\label{9001}\\
X_{i,-j,r}&=\Big(\prod_{u=i}^{j-2}\ad(X^+_{u,0})\Big)\Big(\prod_{u=j+1}^{n}\ad(X_{u-2,u,0})\Big)x^+_{n,r},\label{9007}\\
X_{n,-n,2r+1}&=X^+_{n+1,2r+1},\label{9009}\\
X_{i,-i,2r+1}&=\ad(X^+_{i,0})\Big(\prod_{u=i+2}^{n}\ad(X_{u-2,u,0})\Big)X^+_{n,2r+1}.\label{9011}
\end{align}
We note that $X_{i,j,r}$ corresponds to $f_{i,j,r}$, $X_{i,-j,r}$ corresponds to $X_{i,-i,2r+1}$ corrresponds to $-f_{i,-i,r}$, and $\begin{cases}
-(-1)^rf_{i,j,r}&\text{ if }j-n\text{ is odd},\\
-f_{i,j,r}&\text{ if }j-n\text{ is even and }j\leq n,\\
f_{i,-n,r}&\text{ if }j=n.
\end{cases}$
\begin{Lemma}\label{gat}
The following elements are spanned by $\{X_{u,v,r},X_{u,-u,r},X_{u,-v,r}\mid1\leq u<v\leq n,r\in\mathbb{Z}_{\geq0}\}$:
\begin{gather*}
[X_{i,j,r},X_{a,b,s}],[X^+_{n,s},X_{i,j,r}],[X^+_{p,s},X_{i,-j,r}],[X^+_{n,s},X_{i,-j,r}],[X^+_{p,s},X_{i,-i,2r+1}][X^+_{n,s},X_{i,-i,2r+1}]
\end{gather*}
for $1\leq p\leq n-1$ and $1\leq i<j\leq n,1\leq a<b\leq n$.
\end{Lemma}
The proof of Theorem~\ref{gat} is given in the appendix. For the proof of Lemma~\ref{gat}, we only use the relations \eqref{Eq901-10-0}-\eqref{Eq901-10-2} and \eqref{Eq901-17}-\eqref{Eq903-8.5}. 

Similarly, for $1\leq i<j\leq n$, we define the following elements:
\begin{align}
X_{j,i,r}&=\prod_{u=i}^{j-2}\ad(X^-_{u,0})(X^-_{j-1,r}),\label{9002}\\
X_{-i,j,r}&=\Big(\prod_{u=i}^{j-2}\ad(X^-_{u,0})\Big)\Big(\prod_{u=j+1}^{n}\ad(X_{u,u-2,0})\Big)X^-_{n,r},\label{9008}\\
X_{-n,n,2r+1}&=X^-_{n+1,2r+1},\label{9010}\\
X_{-i,i,2r+1}&=\ad(X^-_{i,0})\Big(\prod_{u=i+2}^{n}\ad(X_{u,u-2,0})\Big)X^-_{n,2r+1}.\label{9012}
\end{align}
Since we only use the relations \eqref{Eq901-10-0}-\eqref{Eq901-10-2} and \eqref{Eq901-17}-\eqref{Eq903-8.5} in the proof of Lemma~\ref{gat},
we obtain the following lemma by the same proof as Lemma~\ref{gat}.
\begin{Lemma}\label{gat-tui}
The following elements are spanned by $\{X_{v,u,r},X_{-u,u,2r+1},X_{-v,u,r}\mid1\leq u<v\leq n,r\in\mathbb{Z}_{\geq0}\}$:
\begin{gather*}
[X_{j,i,r},X_{b,a,s}],[X^-_{n,s},X_{j,i,r}],[X^-_{p,s},X_{-j,i,r}],[X^-_{n,s},X_{-j,i,r}],[X^-_{p,s},X_{-j,i,2r+1}],[X^-_{n,s},X_{-i,i,2r+1}]
\end{gather*}
for $1\leq p\leq n-1$ and $1\leq i<j\leq n,1\leq a<b\leq n$.
\end{Lemma}
\begin{Lemma}\label{indu1}
The following elements are spanned by the set $\{H_{i,r},X^\pm_{n+1,2r+1},X_{p,q,r}\}$:
\begin{equation*}
[H_{i,s},X_{p,q,r}],\ [X^\pm_{i,s},X_{p,q,r}],\ ,\ [X^\pm_{n+1,s},X_{p,q,r}].
\end{equation*}
\end{Lemma}
\begin{proof}
Suppose that $X_{p,q,r}=\prod_{i=1}^l\ad(X^+_{u_i,0})X^+_{u_{l+1},r_{l+1}}$. The $-$ case can be proven in the same way. As for $[X^+_{i,s},X_{p,q,r}]$ follows from Lemma~\ref{gat}.
We can prove other cases by the induction on $l$. We only show the statement for $[H_{i,s},X_{p,q,r}]$. Other cases can be proven by the similar way.

In the case that $l=0$, it is trivial. Suppose that the statement is true if $l\leq m$. For the case that $l=m+1$, by \eqref{Eq901-1}, we can rewrite $[H_{i,s},X_{p,q,r}]$ as a linear sum of the elements whose form is
\begin{equation*}
\prod_{i=1}^l\ad(X^\pm_{u_i,r_i})X^\pm_{u_{l+1},r_{l+1}}.
\end{equation*}
Thus, by the induction hypothesis and Lemma~\ref{gat}, we find that $[H_{i,s},X_{p,q,r}]$ can be spanned by  $X_{p,q,r}$.
\end{proof}
\begin{Theorem}\label{pre}
The universal enveloping algebra $U(\mathfrak{sl}(2n)[u]^\tau)$ is isomorphic to $L$.
\end{Theorem}
\begin{proof}
By replacing $x$ with $X$, we define $x_{p,q,r}$ for $p\neq q$.
There exists a natural surjective homomorphism 
\begin{equation*}
\rho\colon L\to U(\mathfrak{sl}(2n)[u]^\tau)
\end{equation*}
by $X^\pm_{i,r}\mapsto x^\pm_{i,r}$ and $H_{i,r}\mapsto h_{i,r}$. By the definition, we find that $\mathfrak{sl}(2n)[u]^\tau$ has a basis which consists of $h_{i,r}$, $x^\pm_{n+1,2r+1}$ and $x_{p,q,r}$. By the PBW theorem, the ordered monomials of $h_{i,r}$, $x^\pm_{n+1,2r+1}$ and $x_{p,q,r}$ become a basis of $U(\mathfrak{sl}(2n)[u]^\tau)$. Thus, it is enough to show that the ordered monomials of $H_{i,r},X^\pm_{n+1,2r+1}$ and $X^\pm_{p,q,r}$ form the spanning set of $L$.
This follows from Lemma~\ref{indu1} and \eqref{Eq901-1-1}.
\end{proof}
\section{Finite presentation of $U(\mathfrak{sl}(2n)[u]^\tau)$}
One of the difficulty of the presentation given in Theorem~\ref{pre} is that the number of generators is infinite. In this section, we give a presentation of $U(\mathfrak{sl}(2n)[u]^\tau)$ with finite generators.
\begin{Theorem}\label{Mini}
The universal enveloping algebra $U(\mathfrak{sl}(2n)[u]^\tau)$ is isomorphic to the associative algebra generated by
\begin{equation*}
\{H_{i,0},H_{j,1},X^\pm_{i,r}\mid 1\leq i\leq n,1\leq j\leq n-1,r=0,1\}
\end{equation*}
with the commutator relations:
\begin{gather}
[H_{i,r},H_{j,s}]=0\text{ if }0\leq r,s\leq1,\label{eq901}\\
[H_{i,r},X^\pm_{j,0}]=\pm a^{r}_{i,j}X^\pm_{j,r}\text{ for }r=0,1,\label{eq901-1}\\
[X^+_{i,0},X^-_{j,0}]=\delta_{i,j}H_{i,0},\label{eq901-2--1},\\
[X^+_{i,1},X^-_{i,0}]=\delta_{i,j}H_{i,1}\text{ if }i\neq n,\label{eq901-2-0}\\
[[X^+_{n-1,2},X^-_{n-1,0}],X^\pm_{n,0}]=0,\label{eq901-2-1}\\
[X^+_{n,1},X^-_{n,0}]=H_{n-1,1}\label{eq901-3-1}\\
[X^\pm_{i,1},X^\pm_{i,0}]=[X^\pm_{i,0},X^\pm_{i,1}],\label{eq901-10-1}\\
[X^\pm_{i,0},X^\pm_{j,0}]=0\text{ if }a_{i,j}=0,\label{eq901-17}\\
[X^\pm_{i,0},[X^\pm_{i,0},X^\pm_{j,0}]]=0\text{ if }a_{i,j}=-1,\label{eq902}\\
[X^\pm_{n,1},[X^\pm_{n-2,0},X^\pm_{n-1,0}]]=-[X^\pm_{n-1,0},[X^\pm_{n-2,0},X^\pm_{n,1}]],\label{eq903-1}\\
[[X^\pm_{n,0},X^\pm_{n-2,0}],[X^\pm_{n,0},X^\pm_{n-1,1}]]=0,\label{eq903-3}\\
[[X^\pm_{n-1,0},X^\pm_{n-2,0}],[X^\pm_{n-1,0},X^\pm_{n,1}]]=0.\label{eq903-7}
\end{gather}
\end{Theorem}
This section is devoted to the proof of Theorem~\ref{Mini}. We define $X^\pm_{i,s},H_{i,2s-2},H_{j,2s-1}$ for $s\geq 2$, $1\leq i\leq n$ and $1\leq j\leq n-1$ inductively by
\begin{gather*}
X^\pm_{i,s}=\begin{cases}
\dfrac{1}{2}[H_{i,1},X^\pm_{i,s-1}]&\text{ if }1\leq i\leq n-1,\\
\dfrac{1}{2}[H_{n-1,1},X^\pm_{n,s-1}]&\text{ if }i=n,
\end{cases}\\
H_{i,2s-2}=[X^\pm_{i,2s-2},X^-_{i,0}],\ H_{j,2s-1}=[X^+_{j,2s-1},X^-_{j,0}].
\end{gather*}
We also define
\begin{align*}
X^+_{n+1,2r+1}&=-[X^+_{n,2r+1},X^-_{n-1,0}],\ 
X^-_{n+1,2r+1}=-[X^+_{n-1,2r+1},X^-_{n,0}].
\end{align*}
By the definition of $H_{i,2s-2}$, we can rewrite \eqref{eq901-2-1} as
\begin{align}
[H_{n-1,2},X^\pm_{n,0}]=0.\label{eqeq0-1}.
\end{align}
We denote the associative algebra defined by \eqref{eq901}-\eqref{eq903-7} by $\widetilde{L}$.
\begin{Lemma}\label{Lem-1}
There exists a homomorphism from $U(\mathfrak{so}(2n))$ to $\widetilde{L}$ given by $h_i\mapsto H_{i,0}$ and $x^\pm_i\mapsto X^\pm_{i,0}$.
\end{Lemma}
This naturally follows from \eqref{eq901}, \eqref{eq901-1}, \eqref{eq901-2--1}, \eqref{eq901-17} and \eqref{eq902}.

\begin{Lemma}\label{Lem0}
\textup{(1)}\ We obtain
\begin{gather}
[X^+_{i,r},X^-_{i,s}]=\delta_{i,j}H_{i,r+s}\text{ if }r+s\leq 1,1\leq i,j\leq n-1,\label{eq901-2}
\end{gather}
\textup{(2)}\ The following relations hold:
\begin{gather}
[X^+_{n,1},X^-_{n-1,0}]=[X^+_{n,0},X^-_{n-1,1}],\label{eq901-5}\\
[X^+_{n-1,1},X^-_{n,0}]=[X^+_{n-1,0},X^-_{n,1}].\label{eq901-8}
\end{gather}
\textup{(3)}\ We obtain
\begin{gather}
[X^+_{n,r},X^-_{j,s}]=\delta_{n,j}H_{n-1,1}\text{ if }j\neq n-1,n+1\text{ and }r+s=1,\label{eq901-3}\\
[X^+_{j,r},X^-_{n,s}]=\delta_{n,j}H_{n,1},\text{ if }j\neq n-1,n+1\text{ and }r+s=1.\label{eq901-6}
\end{gather}
\textup{(4)}\ The following relation hold:
\begin{gather}
[X^\pm_{i,1},X^\pm_{j,0}]=[X^\pm_{i,0},X^\pm_{j,1}]\text{ for }1\leq i,j\leq n\text{ and }(i,j)\neq(n-1,n),(n,n-1).\label{eq901-10}
\end{gather}
\end{Lemma}
\begin{proof}
\textup{(1)}\ In the case that $1\leq i,j\leq n-1$ and $i\neq j$, by \eqref{eq901-1}, we have
\begin{gather*}
0=[H_{i,1},[X^+_{i,0},X^-_{j,0}]]=a_{i,i}[X^+_{i,1},X^+_{j,0}]-a_{i,j}[X^+_{i,0},X^+_{j,1}],\\
0=[H_{j,1},[X^+_{i,0},X^-_{j,0}]]=a_{i,j}[X^+_{i,1},X^+_{j,0}]-a_{j,j}[X^+_{i,0},X^+_{j,1}].
\end{gather*}
Since the matrix $\begin{pmatrix}
a_{i,i} & a_{i,j} \\
a_{i,j} & a_{j,j} \\
\end{pmatrix}$
is invertible, we find the relation $[X^+_{i,r},X^+_{j,1-r}]=0$. By \eqref{eq901} and \eqref{eq901-1}, we obtain
\begin{gather*}
0=[H_{i,1},H_{i,0}]=[H_{p,1},[X^+_{i,0},X^-_{i,0}]]=a_{i,i}([X^+_{i,1},X^-_{i,0}]-[X^+_{i,0},X^-_{i,1}])
\end{gather*}
Thus, we find that $[X^+_{i,0},X^-_{i,1}]=[X^+_{i,1},X^-_{i,0}]=H_{i,1}$ holds.

We can prove (2) and (3) by the similar way to (1).

\textup{(4)}\ We only consider the $+$ case. The $-$ case can be proven in the same way. We need to show the case that $i\neq j$. In the case that $a_{i,j}=0$, $i\neq n$ and $(i,j)\neq (n-1,n),(n,n-1)$, we can choose $k$ such that $a^1_{i,k}\neq0$ and $a^1_{j,k}=0$. Then, we obtain
\begin{align*}
[X^+_{i,1},X^+_{j,0}]&=\dfrac{1}{a^1_{i,k}}[H_{k,1},[X^+_{i,0},X^+_{j,0}]]=0.
\end{align*}
In the case that $a_{i,j}=-1$, we have
\begin{align*}
0=[[X^+_{i,0},[X^+_{i,0},X^+_{j,0}]],X^-_{i,1}]&=[X^+_{i,0},[H_{i,1},X^+_{j,0}]]+[H_{i,1},[X^+_{i,0},X^+_{j,0}]]\\
&=-[X^+_{i,0},X^+_{j,1}]+2[X^+_{i,1},X^+_{j,0}]-[X^+_{i,0},X^+_{j,1}]
\end{align*}
where we consider $H_{n,1}=H_{n-1,1}$. Thus, we have proved the relation.
\end{proof}

\begin{Lemma}\label{Lem1}
For $1\leq i,j\leq n$ and $(i,j)\neq(n-1,n),(n,n-1)$, the relations \eqref{Eq901}-\eqref{Eq901-3}, \eqref{Eq901-10}, \eqref{Eq901-17} and \eqref{Eq902} are obtained. Moreover, we obtain the relation
\begin{equation}
[H_{n,2r},H_{n-1,2s+1}]=0.\label{eqeq0-0}
\end{equation}
\end{Lemma}
\begin{proof}
In the case that $i,j\in\{1,2,\cdots,n\}\setminus\{n\}$, the relations \eqref{Eq901}-\eqref{Eq901-3}, \eqref{Eq901-10}, \eqref{Eq901-17} and \eqref{Eq902} are nothing but the defining relations of the Yangian associated with $\mathfrak{sl}(n)$. Thus, by Theorem~2.13 in Guay-Nakajima-Wendlandt \cite{GNW}, the relations \eqref{Eq901}-\eqref{Eq901-3}, \eqref{Eq901-10}, \eqref{Eq901-17} and \eqref{Eq902} hold. 

In the case that $i,j\in\{1,2,\cdots,n\}\setminus\{n-1\}$, let us set
\begin{align*}
h_{i,r}=\begin{cases}
H_{n-1,r}&\text{ if }i=n,r\text{ is odd},\\
H_{i,r}&\text{ otherwise}.
\end{cases}
\end{align*}
Then, the generators $\{h_{i,r},X^\pm_{i,r}\mid i\in\{1,\cdots,n\}\setminus\{n-1\}\}$ and the relations \eqref{Eq901}-\eqref{Eq901-3}, \eqref{Eq901-10}, \eqref{Eq901-17} and \eqref{Eq902} define the Yangian associated with $\mathfrak{sl}(n)$. Thus, by Theorem~2.13 in Guay-Nakajima-Wendlandt \cite{GNW}, the relations \eqref{Eq901}-\eqref{Eq901-3}, \eqref{Eq901-10}, \eqref{Eq901-17} and \eqref{Eq902} hold. 
\end{proof}
\begin{Lemma}
The following relations hold for $r+s=2$:
\begin{equation}
[X^+_{n,r},X^-_{n-1,s}]=0,\ [X^-_{n,r},X^+_{n-1,s}]=0.\label{LemLem}
\end{equation}
\end{Lemma}
\begin{proof}
We only show the first equation. The second equation can be proven in a similar way.
By \eqref{eq901-2--1}, Lemma~\ref{Lem1} and \eqref{eqeq0-1}, we have
\begin{align}
0&=[H_{n-1,2},[X^-_{n-1,0},X^+_{n,0}]]=-2[X^-_{n-1,2},X^+_{n,0}].\label{ultra}
\end{align}
We obtain
\begin{align}
0&=[H_{n-2,2},[X^-_{n-1,0},X^+_{n,0}]]=[X^-_{n-1,2},X^+_{n,0}]-[X^-_{n-1,0},X^+_{n,2}]=-[X^-_{n-1,0},X^+_{n,2}]\label{ultra-2}
\end{align}
by \eqref{eq901-2--1}, Lemma~\ref{Lem1} and \eqref{ultra}. By \eqref{ultra-2}, Lemma~\ref{Lem1}, \eqref{eq901-5}, we have
\begin{align*}
&\quad[X^+_{n,1},X^-_{n-1,1}]=[X^+_{n,1},X^-_{n-1,1}]-[X^+_{n,2},X^-_{n-1,0}]\\
&=[H_{n-2,1},[X^+_{n,1},X^-_{n-1,0}]]=[H_{n-2,1},[X^+_{n,0},X^-_{n-1,1}]]\\
&=-[X^+_{n,1},X^-_{n-1,1}]+[X^+_{n,0},X^-_{n-1,2}]=-[X^+_{n,1},X^-_{n-1,1}].
\end{align*}
Thus, we have obtained $[X^+_{n,r},X^-_{n-1,s}]=0$ for $r+s=2$.
\end{proof}
\begin{Corollary}
The following relation holds:
\begin{equation}
[H_{n-2,1},X^\pm_{n+1,1}]=0.\label{eeq901}
\end{equation}
\end{Corollary}
\begin{proof}
We only show the $+$ case. The $-$ case can be proven in a similar way. By \eqref{LemLem}, Lemma~\ref{Lem1} and the definition of $X^+_{n+1,1}$, we obtain
\begin{align}
[H_{n-2,1},X^+_{n+1,1}]=-\dfrac{1}{2}[H_{n-2,1},[X^+_{n,1},X^-_{n-1,0}]]=-[X^+_{n,2},X^-_{n-1,0}]+[X^+_{n,1},X^-_{n-1,1}]=0.\label{muju2}
\end{align}
\end{proof}
\begin{Lemma}
The following relation holds:
\begin{equation}
[H_{n,2},X^\pm_{n-1,r}]=0.\label{eqeq1}
\end{equation}
\end{Lemma}
\begin{proof}
We only show the $+$ case. The $-$ case can be proven by the same way.
First, we show the case that $r=0$. We obtain
\begin{align*}
[H_{n,2},X^+_{n-1,0}]
&=[[X^+_{n,0},X^-_{n,2}],X^+_{n-1,0}]=[X^+_{n,0},[X^-_{n,2},X^+_{n-1,0}]]=0,
\end{align*}
where the first equality is due to Lemma~\ref{Lem1}, the second equality is due to \eqref{eq901-17} and the last equality is due to \eqref{LemLem}.

Since the relation $[H_{n,2},H_{n-2,1}]=0$ holds by Lemma~\ref{Lem1}, we obtain
\begin{equation*}
[H_{n,2},X^+_{n-1,r}]=-[H_{n,2},[H_{n-2,1},X^+_{n,r-1}]]=-[H_{n-2,1},[H_{n-1,2},X^+_{n,r-1}]].
\end{equation*}
Thus, the case that $r\geq1$ can be proven inductively.
\end{proof}
By the same discussion as the last part of the above proof, we obtain
\begin{align}
[H_{n,2},X^\pm_{n-1,r}]=0.\label{eqeq0}
\end{align}
\begin{Lemma}
The following relations hold:
\begin{gather}
[X^+_{n+1,1},X^\pm_{j,0}]=0,\text{ if }j\neq n-1,n,n+1,\label{eq901-11}\\
[X^-_{n+1,1},X^\pm_{j,0}]=0,\text{ if }j\neq n-1,n,n+1,\label{eq901-12}\\
[X^+_{n+1,1},X^+_{n-1,0}]=-2X^+_{n,1},\ [X^+_{n+1,1},X^-_{n-1,0}]=0,\label{eq901-13}\\
[X^+_{n+1,1},X^-_{n,0}]=2X^-_{n-1,1},\ [X^+_{n+1,1},X^+_{n,0}]=0,\label{eq901-14}\\
[X^-_{n+1,1},X^-_{n-1,0}]=2X^-_{n,1},\ [X^-_{n+1,1},X^+_{n-1,0}]=0,\label{eq901-15}\\
[X^-_{n+1,1},X^+_{n,0}]=-2X^+_{n-1,1},\ [X^-_{n+1,1},X^-_{n,0}]=0.\label{eq901-16}
\end{gather}
\end{Lemma}
\begin{proof}
We only show the relation \eqref{eq901-13}. The other relations can be proven by the same way. By the definition of $X^+_{n+1,1}$, we have
\begin{align*}
[X^+_{n+1,1},X^+_{n-1,0}]&=-[[X^+_{n,1},X^-_{n-1,0}],X^+_{n-1,0}]\\
&=-[[X^+_{n,0},X^-_{n-1,1}],X^+_{n-1,0}]=-[X^+_{n,0},H_{n-1,1}]=-2X^+_{n,1},
\end{align*}
where the second equality is due to \eqref{eq901-5}, the third equality is due to \eqref{Eq901-4} and \eqref{eq901-2}, and the last duality is due to \eqref{eq901-1}. Similarly, by the definition of $X^+_{n+1,1}$, we obtain
\begin{align*}
[X^+_{n+1,1},X^-_{n-1,0}]&=-[[X^+_{n,1},X^-_{n-1,0}],X^-_{n-1,0}]\\
&=-[[X^+_{n,0},X^-_{n-1,1}],X^-_{n-1,0}]=0,
\end{align*}
where the second equality is due to \eqref{eq901-5} and third equality is due to \eqref{Eq901-4} and \eqref{eq901-10}.
\end{proof}
\begin{Corollary}\label{Cor}
\begin{enumerate}
\item We obtain \eqref{Eq901-10}, \eqref{Eq901-17-1} and \eqref{Eq901-17-2}.
\item The relations \eqref{Eq901-4}, \eqref{Eq901-5}, \eqref{Eq901-7} and \eqref{Eq901-8} hold. 
\item We have $[H_{n,2},X^\pm_{n+1,2r+1}]=\pm2X^\pm_{n+1,2r+3}=[H_{n-1,2},X^\mp_{n+1,2r+1}]$.
\end{enumerate}
\end{Corollary}
\begin{proof}
\begin{enumerate}
\item We only show the $+$ case. The $-$ case can be proven by the same way. First, we show the relation \eqref{Eq901-17-1}. We prove by the induction on $r+s$. The cases that $r+s=0$ is nothing but \eqref{eq901-17} and \eqref{Eq901-17-1}. We suppose that \eqref{Eq901-17} holds for $r+s=2k$. 
For $r\geq2$, by \eqref{eqeq0}, \eqref{eqeq1} and Lemma~\ref{Lem1}, we have
\begin{align*}
[X^+_{n-1,2k+2-r},X^+_{n,r}]=\dfrac{1}{2}[H_{n,2},[X^+_{n-1,k+2-r},X^+_{n,r-2}]]=0
\end{align*}
and
\begin{align*}
[X^+_{n,2k+2-r},X^+_{n-1,r}]=\dfrac{1}{2}[H_{n-1,2},[X^+_{n,k+2-r},X^+_{n-1,r-2}]]=0.
\end{align*}
Thus, we have proved the relation \eqref{Eq901-17-1}.

Next, we show the relation \eqref{Eq901-10}. By Lemma~\ref{Lem0} and \eqref{Eq901-17-1}, we have
\begin{align*}
0=[H_{n-2,1},[X^\pm_{n-1,r},X^\pm_{n,2k-r}]]=\mp[X^\pm_{n-1,r+1},X^\pm_{n,2k-r}]\mp[X^\pm_{n-1,r},X^\pm_{n,2k-r+1}].
\end{align*}
Thus, we have shown \eqref{Eq901-10}.

Finally, we show the relation \eqref{Eq901-17-2}. Suppose that $(i,j)=(n-1,n),(n,n-1)$ and $r_1+s,r_2+s$ are odd. In the case that $r_2\neq0$, we obtain
\begin{align*}
&\quad[X^\pm_{i,r_1},[X^\pm_{i,r_2},X^\pm_{j,s}]]=-[X^\pm_{i,r_1},[X^\pm_{i,r_2-1},X^\pm_{j,s+1}]]\\
&=-[X^\pm_{i,r_2-1},[X^\pm_{i,r_1},X^\pm_{j,s+1}]]=0,
\end{align*}
where the first equality is due to \eqref{Eq901-10} and the second equality is due to \eqref{Eq901-17}.
In the case that $r_2=0$, by \eqref{Eq901-10} and \eqref{Eq901-17}, we have
\begin{align*}
&\quad[X^\pm_{i,r_1},[X^\pm_{i,r_2},X^\pm_{j,s}]]=-[X^\pm_{i,r_1},[X^\pm_{i,r_2+1},X^\pm_{j,s-1}]]\\
&=-[X^\pm_{i,r_2+1},[X^\pm_{i,r_1},X^\pm_{j,s-1}]]=0.
\end{align*}
Thus, we have proved \eqref{Eq901-17-2}.
\item
We prove \eqref{Eq901-4} by the induction on $r+s$. 
The cases that $r+s=0,2$ are nothing but \eqref{Eq901-4} and \eqref{LemLem}. Suppose that \eqref{Eq901-4} holds for $r+s=2m$. By \eqref{eqeq0} and \eqref{eqeq1}, we have
\begin{align*}
[X^+_{n,r},X^-_{n-1,2m+2-r}]&=-\dfrac{1}{2}[H_{n-1,2},[X^+_{n,r},X^-_{n-1,2m-r}]]=0,\\
[X^+_{n,2m+2-r},X^-_{n-1,r}]&=\dfrac{1}{2}[H_{n,2},[X^+_{n,2m-r},X^-_{n-1,r}]]=0.
\end{align*}
Thus, we have proved \eqref{Eq901-4}. By adopting $\ad(H_{n-2,1})$ to $[X^+_{n,r},X^-_{n-1,2m-r}]=0$, we obtain
\begin{align*}
-[X^+_{n,r+1},X^-_{n-1,2m-r}]+[X^+_{n,r},X^-_{n-1,2m-r+1}]=0.
\end{align*}
Thus, we have proved \eqref{Eq901-5}.
We can prove \eqref{Eq901-7} and \eqref{Eq901-8} similarly to \eqref{Eq901-4} and \eqref{Eq901-5}.
\item We only show the $+$ case. The $-$ case can be proven by the same way. By the definition of $X^\pm_{n+1,2r+1}$ and \eqref{eqeq1}, we have
\begin{align*}
[H_{n,2},X^+_{n+1,2r+1}]&=-[H_{n,2},[X^+_{n,2r+1},X^-_{n-1,0}]]=-2[X^+_{n,2r+3},X^-_{n-1,0}]=2X^\pm_{n+1,2r+3}.
\end{align*}
Similarly, we obtain
\begin{align*}
[H_{n-1,2},X^+_{n+1,2r+1}]&=-[H_{n-1,2},[X^+_{n,2r+1},X^-_{n-1,0}]]\\
&=2[X^+_{n,2r+1},X^-_{n-1,2}]=2[X^+_{n,2r+3},X^-_{n-1,0}]=-2X^+_{n+1,2r+3},
\end{align*}
where the third equality is due to Corollary~\ref{Cor} 2.
\end{enumerate}
\end{proof}
\begin{Lemma}
The relations \eqref{Eq901-11}-\eqref{Eq901-16} hold.
\end{Lemma}
\begin{proof}
We only show \eqref{Eq901-16}. The other relations can be proven in the same way. We prove by the induction on $r+s$. The case that $r+s=0$ is nothing but \eqref{eq901-16}. By \eqref{eeq901}, we have
\begin{align*}
2[X^-_{n+1,2r+1},X^+_{n,s+1}]=[X^-_{n+1,2r+1},[H_{n-1,1},X^+_{n,s}]]=[H_{n-1,1},-2X^+_{n-1,2r+s+1}]=-4X^+_{n-1,2r+s+2}.
\end{align*}
By using \eqref{eqeq1} and Corollary~\ref{Cor} 3, we have
\begin{align*}
2[X^-_{n+1,2r+3},X^+_{n,0}]=[H_{n-1,2},[X^-_{n+1,2r+1},X^+_{n,0}]]=[H_{n-1,2},-2X^+_{n-1,2r+s+1}]=-4X^+_{n-1,2r+s+3}.
\end{align*}
Thus, we have obtained \eqref{Eq901-16}.
\end{proof}
\begin{Lemma}
The relation \eqref{Eq901-1} holds for $(i,j)=(n,n-1),(n-1,n)$.
\end{Lemma}
\begin{proof}
By Lemma~\ref{Lem1}, we obtain
\begin{align*}
[H_{n-1,r},X^+_{n,s}]=\Big(\dfrac{\ad(H_{n-1,1})}{2}\Big)^{s-1}[H_{n-1,r},X^+_{n,0}].
\end{align*}
For the case that $r$ is even, by \eqref{eq901-2--1} and \eqref{Eq901-4}, we have
\begin{align*}
[H_{n-1,r},X^+_{n,0}]=[[X^+_{n-1,r},X^-_{n-1,0}],X^+_{n,0}]=0.
\end{align*}
For the case that $r$ is odd, we obtain
\begin{align*}
[[X^+_{n-1,0},X^-_{n-1,r}],X^+_{n,0}]=[X^+_{n-1,0},X^+_{n+1,r}]=2X^+_{n-1,r},
\end{align*}
where the first equality is due to the definition of $X^+_{n+1,s}$ and \eqref{Eq901-17}, and the second equality is due to \eqref{Eq901-13}.
Thus, by Lemma~\ref{Lem1}, we have obtained \eqref{Eq901-1} for the case that $(i,j)=(n-1,n)$.

Similarly, by Lemma~\ref{Lem1}, \eqref{eq901-2--1} and \eqref{Eq901-4}, we can obtain 
\begin{align*}
[H_{n,2r},X^+_{n-1,s}]&=\Big(\dfrac{\ad(H_{n-1,1})}{2}\Big)^{s-1}[H_{n,2r},X^+_{n-1,0}]\\
&=\Big(\dfrac{\ad(H_{n-1,1})}{2}\Big)^{s-1}[[X^+_{n,2r},X^-_{n,0}],X^+_{n-1,0}]=0.
\end{align*}
\end{proof}
\begin{Corollary}\label{Last}
\begin{enumerate}
\item
The relation \eqref{Eq901} holds for $(i,j)=(n,n-1),(n-1,n)$.
\item The relations \eqref{Eq901-1-1}, \eqref{Eq901-9} and \eqref{Eq901-10} hold.
\end{enumerate}
\end{Corollary}
\begin{proof}
\begin{enumerate}
\item
By \eqref{Eq901-1} and the definition of $H_{i,r}$, we have
\begin{align*}
[H_{n-1,r},H_{n,s}]=[H_{n-1,r},[X^+_{n,s},X^-_{n,0}]]=a^r_{n-1,n}([X^+_{n,s+r},X^-_{n,0}]-[X^+_{n,s},X^-_{n,r}])=0,
\end{align*}
where the last equality is due to Lemma~\ref{Lem1}.
\item First, we show the relation \eqref{Eq901-1-1}. We only prove the $+$ case. The $-$ case can be shown in a similar way. We obtain
\begin{align*}
&\quad[H_{i,2r},X^+_{n+1,2s+1}]\\
&=[H_{i,2r},-\dfrac{1}{2}[X^+_{n,2s+1},X^-_{n-1,0}]]\\
&=-(2\delta_{i,n}-\delta_{i,n-2})[X^+_{n,2s+2r+1},X^-_{n-1,0}]+(2\delta_{i,n-1}-\delta_{i,n-2})[X^+_{n,2s+1},X^-_{n-1,2r}]\\
&=(2\delta_{i,n}-\delta_{i,n-2})X^+_{n+1,2s+2r+1}-(2\delta_{i,n-1}-\delta_{i,n-2})X^+_{n+1,2s+2r+1}\\
&=(2\delta_{i,n}-2\delta_{i,n-1})X^+_{n+1,2r+2s+1},
\end{align*}
where the first equality is due to the definition of $X^+_{n+1,2s+1}$, and the second and third equalities are is due to \eqref{Eq901-1} and the last equality is due to \eqref{Eq901-6}.
Similarly, we also obtain
\begin{align*}
&\quad[H_{i,2r+1},X^+_{n+1,2s+1}]\\
&=[H_{i,2r+1},-\dfrac{1}{2}[X^+_{n,2s+1},X^-_{n-1,0}]]\\
&=-(2\delta_{i,n-1}-\delta_{i,n-2})[X^+_{n,2r+2s+2},X^-_{n-1,0}]+(2\delta_{i,n-1}-\delta_{i,n-2})[X^+_{n,2s+1},X^-_{n-1,2r+1}]\\
&=0-0=0,
\end{align*}
where the first equality is due to the definition of $X^+_{n+1,2s+1}$, the second equality is due to \eqref{Eq901-1} and the third equality is due to \eqref{Eq901-5}.

Next, we show the relation \eqref{Eq901-9}. 
\begin{align*}
&\quad[X^+_{n+1,2r+1},X^-_{n+1,2s+1}]=[-[X^+_{n,2r+1},X^-_{n-1,0}],X^-_{n+1,2s+1}]\\
&=-[[X^+_{n,2r+1},X^-_{n+1,2s+1}],X^-_{n-1,0}]-[X^+_{n,2r+1},[X^-_{n-1,0},X^-_{n+1,2s+1}]]\\
&=-[2X^+_{n-1,2r+2s+2},X^-_{n-1,0}]-[X^+_{n,2r+1},-2X^-_{n,2s+1}]\\
&=2(-H_{n-1,2r+2s+2}+H_{n,2r+2s+2}),
\end{align*}
where the first equality is due to the definition of $X^+_{n+1,2s+1}$, the third equality is due to \eqref{Eq901-15} and \eqref{Eq901-16} and the last equality is due to \eqref{Eq901-2}.

Finally, we prove the relation \eqref{Eq901-10}. We only show the $+$ case. The $-$ case can be proven in the same way.
\begin{align*}
&\quad[X^+_{n+1,2r+1},X^+_{n+1,2s+1}]=[X^+_{n+1,2r+1},-[X^+_{n,2s+1},X^-_{n-1,0}]]\\
&=-[X^+_{n,2s+1},[X^+_{n+1,2r+1},X^-_{n-1,0}]]-[[X^+_{n+1,2r+1},X^+_{n,2s+1}],X^-_{n-1,0}]=0+0=0,
\end{align*}
where the first equality is due to the definition of $X^+_{n+1,2s+1}$ and the second equality is due to \eqref{Eq901-13} and \eqref{Eq901-14}.
\end{enumerate}
\end{proof}
\begin{Lemma}\label{LemLem2}
The relation \eqref{Eq903-1} holds.
\end{Lemma}
\begin{proof}
We only show the $+$ case. The $-$ case can be proven in a similar way.
Let us set
\begin{align*}
f(r,s,u)=[X^+_{n,s},[X^+_{n-2,r},X^+_{n-1,u}]]-(-1)^{s+u}[X^+_{n-1,u},[X^+_{n-2,r},X^+_{n,s}]].
\end{align*}
By \eqref{Eq901-1}, we have
\begin{align}
[H_{n-3,1},f(r,s,u)]=-f(r+1,s,u).\label{last91}
\end{align}
Thus, it is enough to show the case that $r=0$.

We prove this by induction on $s+u$. The case that $s+u=0$ follows from Lemma~\ref{Lem-1}. For the case that $s+u=1$, we obtain
\begin{align*}
0=[H_{n-2,1},f(0,0,0)]=2f(1,0,0)-f(0,1,0)-f(0,0,1)=0-0-f(0,0,1),
\end{align*}
by \eqref{last91}, \eqref{Eq901-1} and \eqref{eq903-1}. Thus, \eqref{Eq903-1} is true for the case that $s+u=1$.
The case that $s+u=2$ follows from the following relations:
\begin{gather*}
0=[H_{n,2},f(r,0,0)]=2f(r,2,0)-f(r+2,0,0)=2f(r,2,0),\\
0=[H_{n-1,2},f(r,0,0)]=2f(r,0,2)-f(r+2,0,0)=2f(r,0,2),\\
[H_{n-2,1},f(r,1,0)]=2f(r+1,1,0)-f(r,2,0)-f(r,1,1),
\end{gather*}
which are deduced from \eqref{last91} and \eqref{Eq901-1}. 

Suppose that \eqref{Eq903-1} holds if $s+u\leq m$. We obtain
\begin{gather*}
0=[H_{n,2},f(r,m-u,u)]=2f(r,m+2-u,u)-f(r+2,m-u,u)=2f(r,m+2-u,u)-0,\\
0=[H_{n-1,2},f(r,s,m-s)]=2f(r,s,m+2-s)-f(r+2,s,m-s)=2f(r,s,m+2-s)-0
\end{gather*}
by \eqref{Eq901-1} and \eqref{last91}. Thus, we have proved \eqref{Eq901-1}.
\end{proof}
\begin{Lemma}\label{LemLem3}
We obtain the relations \eqref{Eq903-3}-\eqref{Eq903-8.5}.
\end{Lemma}
\begin{proof}
First, we show \eqref{Eq903-3}. We only show the $+$ case. The $-$ case can be proven in the same way.
Let us set 
\begin{align*}
f(r,s)=[[X^+_{n,0},X^+_{n-2,r}],[X^+_{n,0},X^+_{n-1,s}]].
\end{align*}
By \eqref{Eq901-1} and \eqref{Eq901-10-0} and \eqref{Eq901-10}, we have
\begin{align}
[H_{n-3,1},f(r,s)]&=-f(r+1,s),\label{last93}\\
[H_{n-1,2},f(r,s)]&=-f(r+2,s)+2f(r,s+2).\label{last94}
\end{align}
We prove this by induction on $r+s$. The case $r+s=0$ is derived from the defining relations of $\mathfrak{so}(2n)$. The case $r+s=1$ follows from \eqref{eq903-3} and 
\begin{align*}
0=[H_{n-3,1},f(0,0)]=-f(1,0),
\end{align*}
which are derived from \eqref{last93}.

Suppose that \eqref{Eq903-9} holds in the case that $r+s\leq m+1$. 
By \eqref{last93} and \eqref{last94}, we have
\begin{align*}
0=[H_{n-1,2},f(r,m-r)]&=-f(r+2,m-r)+2f(r,m+2-r),\\
0=[H_{n-3,1},f(r,m+1-r)]&=-f(r+1,m+1-r).
\end{align*}
Thus, we find that \eqref{Eq903-3} holds for the case $r+s=m+2$. We can obtain \eqref{Eq903-7} similarly.

Next, we show the relation \eqref{Eq903-4}. Let us set
\begin{align*}
g(s,r)&=[[X^+_{n,s},[X^+_{n-2,0},X^+_{n-3,0}]],[[X^+_{n-1,0},X^+_{n-2,0}],X^+_{n,r}]].
\end{align*}
By \eqref{Eq901-1}, we have
\begin{align*}
[H_{n-4,1},g(s,r)]&=-[[X^+_{n,s},[X^+_{n-2,0},X^+_{n-3,1}]],[[X^+_{n-1,0},X^+_{n-2,0}],X^+_{n,r}]].
\end{align*}
Since
\begin{align*}
&\quad[X^+_{n,s},[X^+_{n-2,0},X^+_{n-3,1}]]=[X^+_{n,s},[X^+_{n-2,1},X^+_{n-3,0}]]\\
&=[[X^+_{n,s},X^+_{n-2,1}],X^+_{n-3,0}]=[[X^+_{n,s+1},X^+_{n-2,0}],X^+_{n-3,0}]=[X^+_{n,s+1},[X^+_{n-2,0},X^+_{n-3,0}]]
\end{align*}
holds by \eqref{Eq901-17}, we obtain
\begin{align}
[H_{n-4,1},g(s,r)]=-g(s+1,r).\label{last191}
\end{align}
By \eqref{Eq901-1}, we have
\begin{align*}
&\quad[H_{n-1,1},g(s,r)]\\
&=[2[X^+_{n,s+1},[X^+_{n-2,0},X^+_{n-3,0}]]-[X^+_{n,s},[X^+_{n-2,1},X^+_{n-3,0}]],[[X^+_{n-1,0},X^+_{n-2,0}],X^+_{n,r}]]\\
&\quad+[[X^+_{n,s},[X^+_{n-2,0},X^+_{n-3,0}]],\\
&\qquad\qquad\qquad2[[X^+_{n-1,1},X^+_{n-2,0}],X^+_{n,r}]-[[X^+_{n-1,0},X^+_{n-2,1}],X^+_{n,r}]-[[X^+_{n-1,0},X^+_{n-2,0}],X^+_{n,r+1}]].
\end{align*}
Since
\begin{align}
&\quad2[X^+_{n,s+1},[X^+_{n-2,0},X^+_{n-3,0}]]-[X^+_{n,s},[X^+_{n-2,1},X^+_{n-3,0}]]\\
&=2[X^+_{n,s+1},[X^+_{n-2,0},X^+_{n-3,0}]]-[X^+_{n,s+1},[X^+_{n-2,0},X^+_{n-3,0}]]\\
&=[X^+_{n,s+1},[X^+_{n-2,0},X^+_{n-3,0}]]
\end{align}
holds by \eqref{Eq903-17} and
\begin{align*}
&\quad2[[X^+_{n-1,1},X^+_{n-2,0}],X^+_{n,r}]-[[X^+_{n-1,0},X^+_{n-2,1}],X^+_{n,r}]-[[X^+_{n-1,0},X^+_{n-2,0}],X^+_{n,r+1}]\\
&=[[X^+_{n-1,0},X^+_{n-2,1}],X^+_{n,r}]-[[X^+_{n-1,0},X^+_{n-2,0}],X^+_{n,r+1}]\\
&=(-1)^r[[X^+_{n,r},X^+_{n-2,1}],X^+_{n-1,0}]-[[X^+_{n-1,0},X^+_{n-2,0}],X^+_{n,r+1}]\\
&=(-1)^r[[X^+_{n,r+1},X^+_{n-2,0}],X^+_{n-1,0}]-[[X^+_{n-1,0},X^+_{n-2,0}],X^+_{n,r+1}]\\
&=-[[X^+_{n-1,0},X^+_{n-2,0}],X^+_{n,r+1}]-[[X^+_{n-1,0},X^+_{n-2,0}],X^+_{n,r+1}]\\
&=-2[[X^+_{n-1,0},X^+_{n-2,0}],X^+_{n,r+1}]
\end{align*}
hold by \eqref{Eq903-17} and \eqref{Eq903-1}, we obtain
\begin{align}
[H_{n-1,1},g(s,r)]&=g(s+1,r)-2g(s,r+1).\label{last192}
\end{align}

We prove this by induction on $r+s$. The case $r+s=0$ follows from Lemma~\ref{Lem-1}. Suppose that \eqref{Eq903-4} holds for $r+s\leq m$. By \eqref{last191} and \eqref{last192}, we have
\begin{align*}
[H_{n-4,1},g(r,m+1-r)]&=-g(r+1,m+1-r),\\
[H_{n-1,1},g(r,m+1-r)]&=g(r+1,m+1-r)-2g(r,m+2-r).
\end{align*}
Thus, we find that \eqref{Eq903-4} holds for the case $r+s=m+1$.
We can prove \eqref{Eq903-5}, \eqref{Eq903-8} and \eqref{Eq903-8.5} similarly.
\end{proof}
By Lemma~\ref{Lem1}-Lemma~\ref{LemLem3}, we complete the proof of Theorem~\ref{Mini}.

\section{Twisted Yangian of type $D$}
Let us recall the definition of the twisted Yangian given in \cite{MNO}. The twisted Yangian associated with $\mathfrak{so}(2n)$ is defined as a subalgebra of the Yangian associated with $\mathfrak{gl}(2n)$.
\begin{Definition}
Let $\hbar\in\mathbb{C}$.
The Yangian associated with $\mathfrak{gl}(2n)$ is the associative algebra generated by $\{T^{(r)}_{i,j}\mid r\geq0,i,j\in I_n\}$ with the commutator relation:
\begin{gather*}
T^{(0)}_{i,j}=\delta_{i,j},\ [T^{(r+1)}_{i,j},T^{(s)}_{k,l}]-[T^{(r)}_{i,j},T^{(s+1)}_{k,l}]=\hbar(T^{(r)}_{k,j}T^{(s)}_{i,l}-T^{(s)}_{k,j}T^{(r)}_{i,l}).
\end{gather*}
\end{Definition}
We denote the Yangian associated with $\mathfrak{gl}(2n)$ by $Y_\hbar(\mathfrak{gl}(2n))$.  Let us set $\{S^{(r)}_{i,j}\mid r\geq0,i,j\in I_n\}\subset Y_\hbar(\mathfrak{gl}(2n))$ by $S^{(0)}_{i,j}=\delta_{i,j}$ and
\begin{gather*}
S_{i,j}^{(r)}=T^{(r)}_{i,j}-(-1)^rT^{(r)}_{-j,-i}+\hbar\sum_{s=1}^{r-1}\sum_{a\in I_n}(-1)^sT^{(r-s)}_{i,a}T_{-j,-a}^{(s)}.
\end{gather*}
for $r\geq1$. In particular, we have
\begin{align*}
S^{(1)}_{i,j}&=T^{(1)}_{i,j}-T^{(1)}_{-j,-i},\\
S^{(2)}_{i,j}&=T^{(2)}_{i,j}+T^{(2)}_{-j,-i}-\hbar\sum_{a\in I_n}T^{(1)}_{i,a}T^{(1)}_{-j,-a},\\
S^{(3)}_{i,j}&=T^{(3)}_{i,j}-T^{(3)}_{-j,-i}-\hbar\sum_{a\in I_n}(T^{(2)}_{i,a}T^{(1)}_{-j,-a}-T^{(1)}_{i,a}T^{(2)}_{-j,-a}).
\end{align*}
By the definition of $S^{(r)}_{i,j}$, we obtain the following lemma (See Proposition 3.7 in \cite{MNO}).
\begin{Lemma}
\begin{enumerate}
\item We obtain
\begin{align}
S^{(1)}_{i,j}+S^{(1)}_{-j,-i}&=0,\label{SS5}\\
S^{(2)}_{i,j}-S^{(2)}_{-j,-i}&=-\hbar S^{(1)}_{i,j},\label{SS6}\\
S^{(3)}_{i,j}+S^{(3)}_{-j,-i}&=0.\label{SS6.5}
\end{align}
In particular, we have $S^{(1)}_{i,-i}=0$.
\item The following relations hold:
\begin{align}
[S^{(1)}_{i,j},S^{(r)}_{k,l}]&=\delta_{i,k}S^{(r)}_{i,l}-\delta_{i,l}S^{(r)}_{k,j}-\delta_{k,-i}S^{(1)}_{-j,l}+\delta_{-j,l}S^{(r)}_{k,-i}\text{ for }r=1,2,3,\label{SS1}\\
[S^{(2)}_{i,i},S^{(2)}_{j,j}]&=\hbar S^{(1)}_{j,i}S^{(2)}_{i,j}-\hbar S^{(2)}_{j,i}S^{(1)}_{i,j}-\hbar S^{(1)}_{i,-j}S^{(2)}_{-i,j}+\hbar S^{(2)}_{j,-i}S^{(1)}_{-j,i}\text{ for }i,j>0,\label{SS3}\\
[S^{(2)}_{n-1,n},S^{(2)}_{n,n-1}]&=S^{(3)}_{n-1,n-1}-S^{(3)}_{n,n}+\hbar S^{(1)}_{n,n}S^{(2)}_{n-1,n-1}-\hbar S^{(2)}_{n,n}S^{(1)}_{n-1,n-1}\nonumber\\
&\quad-\hbar S^{(1)}_{n-1,-n}S^{(2)}_{-n,n-1}+\hbar S^{(2)}_{n-1,-n}S^{(1)}_{-n+1,n}.\label{SS4}
\end{align}
\end{enumerate}
\end{Lemma}
\begin{Definition}
The twisted Yangian associated with $\mathfrak{so}(2n)$ is the subalgebra generated by $\{S^{(r)}_{i,j}\mid r\geq0,i,j\in\{\pm1,\pm2,\cdots,\pm n\}\}$.
\end{Definition}
We denote the twisted Yangian associated with $\mathfrak{so}(2n)$ by $\widetilde{TY}_\hbar(\mathfrak{so}(2n))$. Let us set the degree of $\widetilde{TY}_\hbar(\mathfrak{so}(2n))$ by $\text{deg}(S^{(r)}_{i,j})=r-1$ for $r\geq1$. Then, $\widetilde{TY}_\hbar(\mathfrak{so}(2n))$ becomes a filtered algebra. We denote the graded associative algebra of $\widetilde{TY}_\hbar(\mathfrak{so}(2n))$ by $\text{gr}\widetilde{TY}_\hbar(\mathfrak{so}(2n))$.
\begin{Theorem}[Proposition 2.23 in \cite{MNO}]\label{gra}
The graded algebra $\text{gr}\widetilde{TY}_\hbar(\mathfrak{so}(2n))$ is isomorphic to $U(\mathfrak{sl}(2n)[u]^\tau)$.
\end{Theorem}
\section{New presentation of the twisted Yangian of type $D$}
We give a finite presentation of the twisted Yangian of type $D$.
\begin{Definition}\label{fin.pre}
We define $TY_\hbar(\mathfrak{so}(2n))$ as the associative algebra generated by
\begin{equation*}
\{H_{i,0},H_{j,1},X^\pm_{i,r}\mid 1\leq i\leq n,1\leq j\leq n-1,r=0,1\}
\end{equation*}
with the defining relations:
\begin{gather}
[H_{i,r},H_{j,s}]=0\text{ if }0\leq r,s\leq1,\label{q901}\\
[H_{i,0},X^\pm_{j,0}]=\pm a_{i,j}X^\pm_{j,0},\label{q901-1-1}\\
[H_{i,1}-\dfrac{\hbar}{2}H_{i,0}^2,X^\pm_{j,0}]=\pm a_{i,j}X^\pm_{j,1},\label{q901-1}\\
[X^+_{i,0},X^-_{j,0}]=\delta_{i,j}H_{i,0},\label{q901-2--1}\\
[X^+_{i,1},X^-_{i,0}]=H_{i,1}\text{ if }i\neq n\label{q901-2-0}\\
[X^+_{n,1},X^-_{n,0}]=H_{n-1,1}-\dfrac{\hbar}{4}(H_{n-1,0}^2-H_{n,0}^2)+\hbar X^-_{n-1,0}X^+_{n-1,0},\label{q901-3}\\
[X^\pm_{i,1},X^\pm_{i,0}]=\hbar(X^\pm_{i,0})^2,\label{q901-10}\\
[X^\pm_{i,0},X^\pm_{j,0}]=0\text{ if }a_{i,j}=0,\label{q901-17}\\
[X^\pm_{i,0},[X^\pm_{i,0},X^\pm_{j,0}]]=0\text{ if }a_{i,j}=-1,\label{q902}\\
[[X^+_{n,0},X^+_{n-2,0}],[X^+_{n,0},X^+_{n-1,1}]]=\hbar X^+_{n,0}[X^+_{n-2,0},[X^+_{n-1,0},X^+_{n,0}]],\label{q903-3}\\
[[X^-_{n,0},X^-_{n-2,0}],[X^-_{n,0},X^-_{n-1,1}]]=-\hbar [X^-_{n-2,0},[X^-_{n-1,0},X^-_{n,0}]]X^-_{n,0},\label{q903-3-}\\
[[X^+_{n-1,0},X^+_{n-2,0}],[X^+_{n-1,0},X^+_{n,1}]]=-\hbar [X^+_{n-2,0},[X^+_{n-1,0},X^+_{n,0}]]X^+_{n-1,0},\label{q903-7}\\
[[X^-_{n-1,0},X^-_{n-2,0}],[X^-_{n-1,0},X^-_{n,1}]]=\hbar X^-_{n-1,0}[X^-_{n-2,0},[X^-_{n-1,0},X^-_{n,0}]]\label{q903-7-}
\end{gather}
and
\begin{align}
&\quad[[X^+_{n-1,1},X^-_{n-1,1}],X^+_{n,0}]\nonumber\\
&=-2(n-1)\hbar [H_{n-1,1},X^+_{n,0}]-2\hbar X^+_{n,1}H_{n-1,0}-2(n-1)\hbar^2 X^+_{n,0}(H_{n-1,0}+ H_{n,0})\nonumber\\
&\quad+(4n-4)\hbar X^+_{n,1}+6(n-1)\hbar^2 X^+_{n,0},\label{q901-2-1}\\
&\quad[[X^+_{n-1,1},X^-_{n-1,1}],X^-_{n,0}]\nonumber\\
&=-2(n-1)\hbar [H_{n-1,1},X^-_{n,0}]+2\hbar H_{n-1,0}X^-_{n,1}+2(n-1)\hbar^2 (H_{n-1,0}+ H_{n,0})X^-_{n,0}\nonumber\\
&\quad-(4n-4)\hbar X^-_{n,1}-6(n-1)\hbar^2 X^-_{n,0},\label{q901-2-1-}\\
&\quad[X^+_{n,1},[X^+_{n-2,0},X^+_{n-1,0}]]-[[X^+_{n,1},X^+_{n-2,0}],X^+_{n-1,0}]\nonumber\\
&=-(n-1)\hbar[X^+_{n-2,0},[X^+_{n-1,0},X^+_{n,0}]]+\hbar X^+_{n-1,0}[X^+_{n-2,0},X^+_{n,0}]+\hbar X^+_{n,0}[X^+_{n-2,0},X^+_{n-1,0}],\label{q903-1}\\
&\quad[X^-_{n,1},[X^-_{n-2,0},X^-_{n-1,0}]]-[[X^-_{n,1},X^-_{n-2,0}],X^-_{n-1,0}]\nonumber\\
&=-(n-1)\hbar[X^-_{n-2,0},[X^-_{n-1,0},X^-_{n,0}]]+\hbar [X^-_{n-2,0},X^-_{n,0}]X^-_{n-1,0}-\hbar [X^-_{n-2,0},X^-_{n-1,0}]X^-_{n,0}.\label{q903-1-}
\end{align}
\end{Definition}
We can connect $TY_\hbar(\mathfrak{so}(2n))$ with the twisted Yangian $\widetilde{TY}_\hbar(\mathfrak{so}(2n))$.
\begin{Theorem}\label{Main}
There exits a homomorphism $\Phi\colon TY_\hbar(\mathfrak{so}(2n))\to \widetilde{TY}_\hbar(\mathfrak{so}(2n))$ given by
\begin{align*}
\Phi(H_{i,1})&=-S^{(2)}_{i,i}+S^{(2)}_{i+1,i+1}-\dfrac{i+n-1}{2}\hbar(S^{(1)}_{i,i}-S^{(1)}_{i+1,i+1})-\hbar S^{(1)}_{i,i}S^{(1)}_{i+1,i+1}\\
&\quad+\hbar\sum_{a\in I_i}S^{(1)}_{i,a}S^{(1)}_{a,i}-\hbar\sum_{a\in I_i}S^{(1)}_{i+1,a}S^{(1)}_{a,i+1},\\
\Phi(X^+_{i,1})&=\begin{cases}
-S^{(2)}_{i,i+1}-\dfrac{i+n-1}{2}\hbar S^{(1)}_{i,i+1}+\hbar\sum_{a\in I_i}S^{(1)}_{i,a}S^{(1)}_{a,i+1}&\text{ if }1\leq i\leq n-1,\\
-S^{(2)}_{n-1,-n}-(n-1)\hbar S^{(1)}_{n-1,-n}+\hbar\sum_{a\in I_{n-1}}S^{(1)}_{n-1,a}S^{(1)}_{a,-n}&\text{ if }i=n,
\end{cases}\\ 
\Phi(X^-_{i,1})&=\begin{cases}
-S^{(2)}_{i,i+1}-\dfrac{i+n-1}{2}\hbar S^{(1)}_{i,i+1}+\hbar\sum_{a\in I_i}S^{(1)}_{i,a}S^{(1)}_{a,i+1}&\text{ if }1\leq i\leq n-1,\\
-S^{(2)}_{-n,n-1}-(n-1)\hbar S^{(1)}_{-n,n-1}+\hbar\sum_{a\in I_{n-1}}S^{(1)}_{-n,a}S^{(1)}_{a,n-1}&\text{ if }i=n.
\end{cases}
\end{align*}
\end{Theorem}
The proof of Theorem~\ref{Main} is given in the appendix.
\begin{Theorem}
The homomorphism $\Phi$ is an isomorphism.
\end{Theorem}
\begin{proof}
Let us take a grading of $TY_\hbar(\mathfrak{so}(2n))$ by $\text{deg}(H_{i,r})=\text{deg}(X^\pm_{i,r})=r$. Then, $TY_\hbar(\mathfrak{so}(2n))$ becomes a filtered algebra. We denote the graded associated algebra of the twisted Yangian $TY_\hbar(\mathfrak{so}(2n))$ by $grTY_\hbar(\mathfrak{so}(2n))$. Since $\Phi$ is a homomorphism of filtered algebras, $\Phi$ induces a homomorphism 
\begin{gather*}
gr\Phi\colon grTY_\hbar(\mathfrak{so}(2n))\to gr\widetilde{TY}_\hbar(\mathfrak{so}(2n))=U(\mathfrak{sl}(2n)[u]^\tau).
\end{gather*}
By the definition of $\Phi$, we have
\begin{align*}
gr\Phi(H_{i,r})&=H_{i,r},\ gr\Phi(X^\pm_{i,r})=X^\pm_{i,r}.
\end{align*}
Since $U(\mathfrak{sl}(2n)[u]^\tau)$ is generated by $H_{i,r},X^\pm_{i,r}$, $gr\Phi$ is surjective. By the definition of $TY_\hbar(\mathfrak{so}(2n))$, \eqref{eq901}-\eqref{eq903-7} hold in $grTY_\hbar(\mathfrak{so}(2n))$. By Theorem~\ref{Mini}, $gr\Phi$ is injective.
Since $gr\Phi$ is an isomorphism, $\Phi$ also becomes an isomorphism.
\end{proof}

\section{Suggestion of a definition of the twisted affine Yangian of type $D$}
Guay \cite{Gu1} defined the affine Yangian associated with $\widehat{\mathfrak{sl}}(n)$ by extending the minimalistic presentation of the Yangian associated with $\mathfrak{sl}(n)$. Similarly, we expect that the twisted affine Yangian of type $D$ can be defined by extending the presentation given in Theorem~\ref{fin.pre}. 

The Cartan matrix of $\widehat{\mathfrak{so}}(2n)$ is given by $A_{i,j}=\begin{cases}
2\delta_{i,j}-\delta_{i,j+1}-\delta_{i+1,j}&\text{ if }1\leq i,j\leq n-1,\\
2\delta_{i,j}-\delta_{i,n-2}-\delta_{j,n-2}&\text{ if }i\text{ or }j=n,\\
2\delta_{i,j}-\delta_{i,2}-\delta_{j,2}&\text{ if }i\text{ or }j=0
\end{cases}$
for $0\leq i,j\leq n$.
\begin{Definition}
We define $TY_\hbar(\widehat{\mathfrak{so}}(2n))$ as the associative algebra generated by
\begin{equation*}
\{H_{i,0},H_{j,1},X^\pm_{i,r}\mid 0\leq i\leq n,1\leq j\leq n-1,r=0,1\}
\end{equation*}
with the defining relations \eqref{q901}, \eqref{q901-2--1},\eqref{q901-3}, \eqref{q901-10}, \eqref{q901-17}-\eqref{q903-1-} and
\begin{gather}
[H_{i,0},X^\pm_{j,0}]=\pm A_{i,j}X^\pm_{j,0},\label{Q901-1-1}\\
[H_{i,1}-\dfrac{\hbar}{2}H_{i,0}^2,X^\pm_{j,0}]=\pm A_{i,j}X^\pm_{j,1},\label{Q901-1}\\
[X^+_{i,1},X^-_{i,0}]=H_{i,1}\text{ if }i\neq 0,n\label{Q901-2-0}\\
[X^+_{0,1},X^-_{0,0}]=H_{1,1}-\dfrac{\hbar}{4}(H_{1,0}^2-H_{0,0}^2)+\hbar X^-_{1,0}X^+_{1,0},\label{Q901-3}\\
[[X^+_{0,0},X^+_{2,0}],[X^+_{0,0},X^+_{1,1}]]=\hbar X^+_{0,0}[X^+_{2,0},[X^+_{1,0},X^+_{0,0}]],\label{Q903-3}\\
[[X^-_{0,0},X^-_{2,0}],[X^-_{0,0},X^-_{1,1}]]=-\hbar [X^-_{2,0},[X^-_{1,0},X^-_{0,0}]]X^-_{0,0},\label{Q903-3-}\\
[[X^+_{1,0},X^+_{2,0}],[X^+_{1,0},X^+_{0,1}]]=-\hbar [X^+_{2,0},[X^+_{1,0},X^+_{0,0}]]X^+_{1,0},\label{Q903-7}\\
[[X^-_{1,0},X^-_{2,0}],[X^-_{1,0},X^-_{0,1}]]=\hbar X^-_{1,0}[X^-_{2,0},[X^-_{1,0},X^-_{0,0}]]\label{Q903-7-}
\end{gather}
and
\begin{align}
&\quad[[X^+_{1,1},X^-_{1,1}],X^+_{0,0}]\nonumber\\
&=-2\hbar [H_{1,1},X^+_{0,0}]-2\hbar X^+_{0,1}H_{1,0}-2\hbar^2 X^+_{0,0}(H_{1,0}+ H_{0,0})\nonumber\\
&\quad+4\hbar X^+_{0,1}+6\hbar^2 X^+_{0,0},\label{Q901-2-1}\\
&\quad[[X^+_{1,1},X^-_{1,1}],X^-_{0,0}]\nonumber\\
&=-2\hbar [H_{1,1},X^-_{0,0}]+2\hbar H_{1,0}X^-_{0,1}+2\hbar^2 (H_{1,0}+ H_{0,0})X^-_{0,0}\nonumber\\
&\quad-4\hbar X^-_{0,1}-6\hbar^2 X^-_{0,0},\label{Q901-2-1-}\\
&\quad[X^+_{0,1},[X^+_{2,0},X^+_{1,0}]]-[[X^+_{0,1},X^+_{2,0}],X^+_{1,0}]\nonumber\\
&=-\hbar[X^+_{2,0},[X^+_{1,0},X^+_{0,0}]]+\hbar X^+_{1,0}[X^+_{2,0},X^+_{0,0}]+\hbar X^+_{0,0}[X^+_{2,0},X^+_{1,0}],\label{Q903-1}\\
&\quad[X^-_{0,1},[X^-_{2,0},X^-_{1,0}]]-[[X^-_{0,1},X^-_{2,0}],X^-_{1,0}]\nonumber\\
&=-\hbar[X^-_{2,0},[X^-_{1,0},X^-_{0,0}]]+\hbar [X^-_{2,0},X^-_{0,0}]X^-_{1,0}-\hbar [X^-_{2,0},X^-_{1,0}]X^-_{0,0}.\label{Q903-1-}
\end{align}
\end{Definition}
The Dynkin diagram of $\widehat{\mathfrak{so}}(2n)$ can be regarded as a combination of two Dynkin diagrams of $\mathfrak{so}(2n)$. Thus, we can propose the following conjecture.
\begin{Conjecture}
The associative algebra $TY_\hbar(\widehat{\mathfrak{so}}(2n))$ becomes a coideal of the affine Yangian associated with $\widehat{\mathfrak{sl}}(n)$ and can be regarded as the twisted affine Yangian of type $D$.
\end{Conjecture}
\appendix
\section{Proof of Lemma~\ref{gat}}
This section is devoted to the proof of Lemma~\ref{gat}. As for $[X_{i,j,r},X_{p,q,s}]$, this follows from the following formula:
\begin{gather}
[X_{i,j,r},X_{p,q,s}]=\delta_{j,p}X_{i,q,r+s}-\delta_{i,q}X_{p,j,r+s}.\label{gat1}
\end{gather}
We note that the relations \eqref{Eq901-10-0}, \eqref{Eq901-17} and \eqref{Eq902} coincide with the defining relations of $\mathfrak{n}_+(n)$. Since $x_{i,j,r}$ corresponds to $E_{i,j}u^r\in\mathfrak{n}_+(n)$, we find that $[X_{i,j,r},X_{p,q,s}]$ is spanned by the given set.

As for other elements, we need to prepare several formulas. 
\begin{Lemma}
The following relations hold:
\begin{gather}
[[X^+_{i,0}, X^+_{j,r}], [X^+_{i,0}, X^+_{k,s}]] = 0 \text{ if } a_{i,j} = a_{i,k} = -1\text{ and }(j,k)\neq (n-1,n),(n,n-1),\label{Eq903-13}\\
[X^+_{n-1,s},[X^+_{n-2,r},X^+_{n,u}]]=(-1)^{r+u}[X^+_{n,r+s+u},[X^+_{n-2,0},X^+_{n-1,0}]],\label{Eq903-16}\\
[X^+_{n-2,r},[X^+_{n-1,s},X^+_{n,u}]]=(1-(-1)^{s+u})(-1)^{r+s}[[X^+_{n-2,0},X^+_{n-1,0}],X^+_{n,r+s+u}],\label{Eq903-17}\\
[[X^+_{n-2,s},X^+_{n-1,u}],X^+_{n,r}]=(-1)^{s+u}[[X^+_{n-2,0},X^+_{n-1,0}],X^+_{n,r+s+u}].\label{Eq903-19}
\end{gather}
\end{Lemma}
\begin{proof}
The relation \eqref{Eq903-13} naturally follows from \eqref{Eq901-17} and \eqref{Eq902}.
We only prove the relation \eqref{Eq903-16} since
\eqref{Eq903-17} and \eqref{Eq903-19} can be proven in a similar way to \eqref{Eq903-16}.
The following relation holds:
\begin{align*}
&\quad[X^+_{n-1,s},[X^+_{n-2,r},X^+_{n,u}]]=[X^+_{n-1,s},[X^+_{n-2,r+u},X^+_{n,0}]]=(-1)^s[X^+_{n,0},[X^+_{n-2,r+u},X^+_{n-1,s}]]\\
&=(-1)^s[X^+_{n,0},[X^+_{n-2,r+s+u},X^+_{n-1,0}]]=(-1)^s[X^+_{n-1,0},[X^+_{n-2,r+s+u},X^+_{n,0}]]\\
&=(-1)^s[X^+_{n-1,0},[X^+_{n-2,0},X^+_{n,r+s+u}]]=(-1)^{r+u}[X^+_{n,r+s+u},[X^+_{n-2,0},X^+_{n-1,0}]],
\end{align*}
where the 1st, 3rd and 5th equalities are due to \eqref{Eq901-10-0} and the 2nd, 4th and 6th equalities are due to \eqref{Eq903-1}.
\end{proof}
\begin{Lemma}
The following relations hold:
\begin{gather}
[X_{i,j,r},X^+_{n,s}]=0\text{ if }j\leq n-2,\label{Eq903-9.5}\\
[X_{i,n-1,0},[X^+_{n-2,r},X^+_{n,s}]]=0,\label{Eq903-12}\\
[X_{i,n-1,0},[X_{i,n,0},X^+_{n,2r}]]=0\text{ for }i<n-1,\label{Eq903-9}\\
[X_{j,n,s},X^+_{n,r}]=
(-1)^s[X_{j,n,0},X^+_{n,r+s}]\text{ for }j\leq n-1,\label{Eq903-10}\\
[X_{i,n-1,s},X^+_{n,r}]=[X_{i,n-1,0},X^+_{n,r+s}]\text{ for }i<n-1,\label{Eq903-11}\\
[X_{i,n-1,s},[X_{j,n-1,0},X^+_{n,r}]]=0\text{ for }i\leq j<n-1,\label{Eq903-14-0}\\
[X_{i,n,s}, [X_{j-1,n-1,0}, X^+_{n,r}]]
    = -(-1)^r[X_{i,n-1,s}, [X_{j-1,n,0}, X^+_{n,r}]]\text{ for }i<n-1,j<n, \label{Eq903-14} \\
[X_{i,n,s}, [X_{j,n,0}, X^+_{n,r}]] = 0 \text{ for } j < n-1, \label{Eq903-18}\\
[X_{i,n,s},[X_{j,n-1,0},X^+_{n,r}]]=(-1)^s[X_{i,n,0},[X_{j,n-1,0},X^+_{n,r+s}]]\text{ for }i,j<n-1,\label{Eq903-20}\\
[X_{i,n-1,s},[X_{j,n,0},X^+_{n,r}]]=[X_{i,n-1,0},[X_{j,n,0},X^+_{n,r+s}]]\text{ for }i,j<n-1,\label{Eq903-21}\\
[X_{i,n,0}, [X^+_{n-2,s}, X^+_{n,r}]]=0\text{ for }i<n-2,\label{Eq903-22}\\
[X^+_{n-2,s}, [X_{i,n-1,0}, X^+_{n,r}]]=0\text{ for }i<n-2,\label{Eq903-23}\\
[X_{i,n,0}, [X_{i,n,0}, X^+_{n,r}]] = 0 \text{ for } i \leq n-1. \label{Eq903-18-2}
\end{gather}
\end{Lemma}
\begin{proof}
First, we show the relation \eqref{Eq903-9.5}. By the definition of $X_{i,j,r}$, we can rewrite $X_{i,j,r}=\prod_{u=i}^{j-2}\ad(X^+_{u,0})X^+_{j-1,r}$. Thus, \eqref{Eq903-9.5} follows from \eqref{Eq901-17} and the assumption that $j\leq n-2$.

Next, we show the relation \eqref{Eq903-12}. We obtain
\begin{align*}
&\quad[X_{i,n-1,0},[X^+_{n-2,r},X^+_{n,s}]]=[[X_{i,n-3,0},X_{n-3,,n-1,0}],[X^+_{n-2,r},X^+_{n,s}]]\\
&=[X_{i,n-3,0},[[X^+_{n-3,0},X^+_{n-2,0}],[X^+_{n-2,r},X^+_{n,s}]]=0,
\end{align*}
where the first equality is due to \eqref{gat1}, the second equality is due to \eqref{gat1} and \eqref{Eq903-9.5}, and the third equality is due to \eqref{Eq903-13} and \eqref{Eq901-10}.

Next, we show the relation \eqref{Eq903-9}. The following relation holds:
    \begin{align}
        &\phantom{{}={}} [X_{i,n-1,0},[X_{i,n,0}, X^+_{n,2r}]]
        = [X_{i,n-1,0}, [[X_{i,n-1,0}, X^+_{n-1,0}],X^+_{n,2r}]]\nonumber \\
        &{}= [X_{i,n-1,0}, [[X_{i,n-1,0},X^+_{n,2r}], X^+_{n-1,0}]]\nonumber \\
        &{}= [[X_{i,n-1,0}, [X_{i,n-1,0}, X^+_{n,2r}]], X^+_{n-1,0}]
            + [[X_{i,n-1,0}, X^+_{n,2r}], [X_{i,n-1,0}, X^+_{n-1,0}]],\label{first1}
    \end{align}
where the first equality is due to \eqref{gat1}, the second equality is due to \eqref{Eq901-17-1} and the third equality is due to the Jacobi identity.
Here after, we denote the $i$-th term of the right hand side of the equality $(\cdot)$ by $(\cdot)_i$. We obtain
\begin{align*}
\eqref{first1}_1&=
[[X^+_{n-2,0}, [X^+_{n-2,0}, X^+_{n,2r}]], X^+_{n-1,0}]=0.
\end{align*}
by \eqref{Eq902} if $i=n-2$ and
\begin{align*}
\eqref{first1}_1
&=[[X_{i,n-1,0}, [[X_{i,n-2,0},X^+_{n-2,0}], X^+_{n,2r}]], X^+_{n-1,0}]\\
&=[[X_{i,n-2,0},[X_{i,n-1,0}, [X^+_{n-2,0}, X^+_{n,2r}]]], X^+_{n-1,0}]=0
\end{align*}
if $i<n-2$, where the first equality is due to \eqref{gat1}, the second equality is due to the assumption of $i$ and \eqref{Eq903-9.5}, and the last equality is due to \eqref{Eq903-12}.
By \eqref{gat1}, we have
\begin{align*}
\eqref{first1}_2&=[[X_{i,n-1,0}, X^+_{n,2r}], X_{i,n,0}]=-[X_{i,n-1,0},[X_{i,n,0}, X^+_{n,2r}]].
\end{align*}
Hence, we obtain the relation $[X_{i,n-1,0},[X_{i,n,0}, X^+_{n,2r}]] = -[X_{i,n-1,0},[X_{i,n,0}, X^+_{n,2r}]]$, which means that $[X_{i,n-1,0},[X_{i,n,0}, X^+_{n,2r}]]=0$.

Next, we show the relation \eqref{Eq903-10}. The cases that $j=n-2,n-1$ directly follows from \eqref{Eq903-19} and \eqref{Eq901-10}. Thus, it is enugh to show the case that $j \leq n-3$. 
The following relation holds for $j\leq n-3$:
    \begin{align*}
        &\phantom{{}={}} [X_{j,n,s}, X^+_{n,r}]
        = [[X_{j,n-2,0}, [X^+_{n-2,0}, X^+_{n-1,s}]], X^+_{n,r}] \\
        &{}= (-1)^s[X_{j,n-2,0}, [[X^+_{n-2,0}, X^+_{n-1,0}], X^+_{n,r+s}]]
        = (-1)^s[X_{j,n,0}, X^+_{n,r}],
    \end{align*}
where the first equality is due to \eqref{gat1} and the second equality is due to \eqref{Eq903-9.5} and \eqref{Eq903-19}.
We can prove \eqref{Eq903-11} by the similar way to \eqref{Eq903-10}.

Next, we show \eqref{Eq903-14-0}. In the case that $i=j=n-2$, the relation follows from \eqref{Eq902}. In the case that $i<j=n-2$, we have
\begin{align*}
&\quad[X_{i,n-1,s},[X_{j,n-1,0},X_{n,r}]]=[X_{i,n-1,s},[X^+_{n-2,0},X_{n,r}]]\\
&=[[X_{i,n-3,0},X_{n-3,n-1,s}],[X^+_{n-2,0},X_{n,r}]]=[X_{i,n-3,0},[[X^+_{n-3,0},X^+_{n-2,0}],[X^+_{n-2,0},X^+_{n,r}]]]=0,
\end{align*}
where the second equality is due to \eqref{gat1}, the third equality is due to \eqref{gat1} and \eqref{Eq903-9.5} and the last equality is due to \eqref{Eq903-13}.

Suppose $i\leq j<n-2$. Then, we obtain
\begin{align*}
&\quad[X_{i,n-1,s},[X_{j,n-1,0},X_{n,r}]]=[X_{i,n-1,s},[[X_{j,n-2,0},X^+_{n-1,0}],X_{n,r}]]\\
&=[X_{j,n-2,0},[X_{i,n-1,s},[X^+_{n-1,0},X_{n,r}]]]=[X_{j,n-2,0},[[X_{i,n-2,0},X^+_{n-1,s}],[X^+_{n-1,0},X_{n,r}]]]\\
&=[X_{j,n-2,0},[X_{i,n-2,0},[X^+_{n-1,s},[X^+_{n-1,0},X_{n,r}]]]]=0,
\end{align*}
where the first and third equalities are due to \eqref{gat1}, the second and 4th equalities are due to \eqref{gat1} and \eqref{Eq903-9.5}, and the last equality is due to \eqref{Eq901-17-1} and \eqref{Eq901-17-2}.

Next, we show \eqref{Eq903-14}. By \eqref{gat1} and the Jacobi identity, we obtain
    \begin{align}
        &\phantom{{}={}} [X_{i,n,0}, [X_{j-1,n-1,0}, X^+_{n,r}]] 
        = [[X_{i,n-1,0}, X^+_{n-1,0}], [X_{j-1,n-1,0}, X^+_{n,r}]]\nonumber\\
        &{}= [[X_{i,n-1,0}, [X_{j-1,n-1,0}, X^+_{n,r}]], X^+_{n-1,0}]-[X_{i,n-1,0}, [[X_{j-1,n-1,0}, X^+_{n,r}], X^+_{n-1,0}]].\label{first2}
    \end{align}
By \eqref{Eq903-14-0}, we find that $\eqref{first2}_1$ is $0$. We also obtain
\begin{align*}
\eqref{first2}_2&=-[X_{i,n-1,0}, [[[X_{j-1,n-2,0},X^+_{n-2,0}], X^+_{n,r}], X^+_{n-1,0}]]\\
&=-[X_{i,n-1,0},[X_{j-1,n-2,0},[[X^+_{n-2,0},X^+_{n,r}],X^+_{n-1,0}]]]\\
&=-(-1)^r[X_{i,n-1,0},[X_{j-1,n-2,0},[[X^+_{n-2,0},X^+_{n-1,0}],X^+_{n,r}]]]\\
&=-(-1)^r[X_{i,n-1,0},[X_{j-1,n,0},X^+_{n,r}]],
\end{align*}
where the first equality is due to \eqref{gat1}, the second equality and last equalities are due to \eqref{Eq903-9.5} and \eqref{gat1}, and the third equality is due to \eqref{Eq903-1}.

Next, we show the relation \eqref{Eq903-18}. In the case that $i=j=n-2$, \eqref{Eq903-18} follows from \eqref{Eq903-8.5} and \eqref{gat1}. We will show other cases.

By \eqref{gat1} and \eqref{Eq903-10}, we have
\begin{align*}
[X_{i,n,s},[X_{j,n,0},X^+_{n,r}]]&=[X_{j,n,0},[X_{i,n,s},X^+_{n,r}]]=(-1)^s[X_{j,n,0},[X_{i,n,0},X^+_{n,r+s}]].
\end{align*}
Thus, it is enough to show the case that $s=0$ and $i\leq j$. 
By \eqref{Eq901-17} and \eqref{gat1}, we have
\begin{align*}
[X_{i,n,0},[X_{j,n,0},X^+_{n,r}]]=\prod_{u=i}^{j-2}\ad(X^+_{u,0})[X_{j-1,n,0},[X_{j,n,0},X^+_{n,r}]]
\end{align*}
if $i+2\leq j$ and
\begin{align*}
[X_{i,n,0},[X_{i,n,0},X^+_{n,r}]]=\ad(X^+_{i,0})[X_{i,n,0},[X_{i+1,n,0},X^+_{n,r}]]
\end{align*}
if $i=j<n-2$. Thus, it is enough to show \eqref{Eq903-17} by induction on $j$ only for the case $i = j - 1$. The case $j = n-2$ immediately follows from \eqref{Eq903-8}. 
    For $j = n-3$, we obtain
    \begin{align*}
        &\phantom{{}={}} [X_{n-4,n,s}, [X_{n-3,n,0}, X^+_{n,r}]] \\
        &{}= [[X_{n-4,n-2,s}, X_{n-2,n,0}], [[X_{n-3,n-1,0}, X^+_{n-1,0}], X^+_{n,r}]] \\
        &{}= [X_{n-4,n-2,s}, [[X_{n-3,n-1,0}, X^+_{n-1,0}], [X_{n-2,n,0}, X^+_{n,r}]]] = 0,     
    \end{align*}
    where the first and second equalities are due to \eqref{gat1}, the third equality is due to \eqref{gat1} and \eqref{Eq903-9.5}, and the last equality is due to \eqref{Eq903-8}.
    
    If $ [X_{j+1,n,0}, [X_{j+2,n,0}, X^+_{n,r}]] = 0 $ holds for $j+2 < n-1$, then we have 
    \begin{align*}
        &\phantom{{}={}} [X_{j-1,n,s}, [X_{j,n,0}, X^+_{n,r}]] \\
        &{}= [[X_{j-1,j+1,s}, X_{j+1,n,0}], [X_{j,j+2,0},[X_{j+2,n,0}, X^+_{n,r}]]] \\
        &{}= [X_{j-1,j+1,s}, [X_{j+1,n,0}, [X_{j,j+2,0}, [X_{j+2,n,0}, X^+_{n,r}]]]] \\
        &{}= [X_{j-1,j+1,s}, [X_{j,j+2,0}, [X_{j+1,n,0}, [X_{j+2,n,0}, X^+_{n,r}]]]] = 0,
    \end{align*}
where the first equality is due to \eqref{gat1}, the second equality is due to \eqref{gat1} and \eqref{Eq903-9.5}, and the last equality is due to the induction hypothesis.

Next, we show the relation \eqref{Eq903-20}. The following relations hold:
\begin{align*}
&\quad[X_{i,n,s},[X_{j,n-1,0},X^+_{n,r}]]=[X_{j,n-1,0},[X_{i,n,s},X^+_{n,r}]]\\
&=(-1)^s[X_{j,n-1,0},[X_{i,n,0},X^+_{n,r+s}]]=(-1)^s[X_{i,n,0},[X_{j,n-1,0},X^+_{n,r+s}]],
\end{align*}
where the first equality is due to \eqref{gat1}, the second equality is due to \eqref{Eq903-10}, and the last equality is due to \eqref{gat1}. The relation \eqref{Eq903-21} can be proven in a similar way.

Next, we show the relation \eqref{Eq903-22} and \eqref{Eq903-23}. 
The following relation holds:
\begin{align*}
&\quad[X_{i,n-1,0},[X^+_{n-2,s},X^+_{n,r}]]=[[X_{i,n-3,0},X_{n-3,n-1,0}],[X^+_{n-2,s},X^+_{n,r}]]\\
&=[X_{i,n-3,0},[X_{n-3,n-1,0},[X^+_{n-2,s},X^+_{n,r}]]]=0,
\end{align*}
where the first equality is due to \eqref{gat1}, the second equality is due to \eqref{gat1} and \eqref{Eq903-9.5}, and the last equality is due to \eqref{Eq903-13}.
By \eqref{Eq902}, we also obtain 
\begin{align*}
&\quad0=[[X_{i,n-1,0},[X^+_{n-2,s},X^+_{n,r}]],X^+_{n-1,0}]=[X_{i,n,0},[X^+_{n-2,s},X^+_{n,r}]].
\end{align*}

Finally, we show \eqref{Eq903-18-2} by induction on $i$. 
The case $i = n-1,n-2$ follows from \eqref{Eq901-17-1},\eqref{Eq901-17-2} and \eqref{Eq903-8.5}. 

    Suppose $[X_{i+1,n,0}, [X_{i+1,n,0}, X^+_{n,r}]] = 0$ holds for $i + 1 \leq n-2$.
    From \eqref{gat1}, we have 
    $X_{i,n,0} = [X_{i,i+1,0}, X_{i+1,n,0}]$ and 
    $[X_{i,n,0}, X_{i,i+1,0}] = 0$. 
    By the Jacobi identity, we have
    \begin{align*}
        &\phantom{{}={}} [X_{i,n,0}, [X_{i,n,0}, X^+_{n,r}]]
        = [X_{i,n,0}, [[X_{i,i+1,0}, X_{i+1,n,0}], X^+_{n,r}]] \\
        &{}= [X_{i,n,0}, [X_{i,i+1,0}, [X_{i+1,n,0}, X^+_{n,r}]]]
        = [X_{i,i+1,0}, [X_{i,n,0}, [X_{i+1,n,0}, X^+_{n,r}]]]
        = 0
    \end{align*}
    using \eqref{Eq903-9.5} and \eqref{Eq903-18} in the second and last equality respectively. 
\end{proof}
We also need to prepare a new presentation of $X_{i,-j,r}$ and $X_{i,-i,2r+1}$.
By \eqref{gat1} and the definition of $X_{i,-j,r}$, we have
\begin{align*}
    X_{i,-j,r} =
    \begin{cases}
        [X_{i,n,0},[X_{j,n-1,0},X^+_{n,r}]] & \text{ if } j\leq n-3 \text{ and } j-n \text{ is odd}, \\
        [X_{i,n-1,0},[X_{j,n,0},X^+_{n,r}]] & \text{ if } j\leq n-2 \text{ and } j-n \text{ is even}, \\
        [X_{i,n,0},X^+_{n,r}] & \text{ if } j=n-1, \\
        [X_{i,n-1,0},X^+_{n,r}] & \text{ if } j=n \text{ and } i < n-1
    \end{cases}
\end{align*}
and
\begin{align*}
    X_{i,-i,2r+1} = 
    \begin{cases}
        [X_{i,n,0}, [X_{i,n-1,0}, X^+_{n,2r+1}]] & \text{ if } i \leq n-2, \\
        [X^+_{n-1,0}, X^+_{n,2r+1}] & \text{ if } i = n-1.
    \end{cases}
\end{align*}
Now, we can complete the proof of Lemma~\ref{gat}
\subsection{Proof for $[X^+_{n,s},X_{i,j,r}]$}
From \eqref{Eq901-10-0}, \eqref{Eq901-10} and \eqref{Eq903-1}, we have
\begin{align*}
    &\phantom{{}={}} \left[ X^+_{n,s}, X_{i,j,r} \right]\\
    &{}= \delta_{j,n-1} \left(\prod_{u=i}^{n-3} \ad(X^+_{u,0}) \right)
    \left[X^+_{n,s}, X^+_{n-2,r}\right]
    + \delta_{j,n} \left( \prod_{u=i}^{n-3} \ad(X^+_{u,0}) \right)
    \left[X^+_{n,s} \left[X^+_{n-2,0}, X^+_{n-1,r}\right]\right] \\
    &{}= \delta_{j,n-1} \left(\prod_{u=i}^{n-3} \ad(X^+_{u,0}) \right)
    \left[X^+_{n,s}, X^+_{n-2,r}\right]\\
  &{}-(-1)^r \delta_{j,n} \left( \prod_{u=i}^{n-3} \ad(X^+_{u,0}) \right)
    \left[\left[X^+_{n-2,0}, X^+_{n-1,0}\right],X^+_{n-2,r+s}]\right] \\
    &{}= -\delta_{j,n-1} X_{i,-n,r+s}
    -(-1)^r \delta_{j,n} X_{i,-n+1,r+s},
\end{align*}
where the first equality is due to \eqref{Eq901-17} and the second equality is due to \eqref{Eq903-19}. Thus, $[X^+_{n,s},X_{i,j,r}]$ can be spanned by the given set.

\subsection{Proof for $[X^+_{p,s},X_{i,-j,r}]$}
We show the statement for each of the cases $p \leq n-3$, $p = n-2$ and $p = n-1$.

First, we show the case that $p\leq n-3$. Suppose that $p \leq n-3$. 
For the case $j \leq n-3$ and $j - n$ is odd, we have 
\begin{align*}
    &\phantom{{}={}} [X^+_{p,s}, X_{i,-j,r}]
    = [X_{p,p+1,s}, [X_{i,n,0}, [X_{j,n-1,0}, X^+_{n,r}]]] \nonumber\\
    &{}= \delta_{p+1, i} [X_{i-1,n,s}, [X_{j,n-1,0}, X^+_{n,r}]]
        + \delta_{p+1, j} [X_{i,n,0}, [X_{j-1,n-1,s}, X^+_{n,r}]]\\
    &{}= \delta_{p+1, i} [X_{i-1,n,s}, [X_{j,n-1,0}, X^+_{n,r}]]
        + \delta_{p+1, j} [X_{i,n,0}, [X_{j-1,n-1,0}, X^+_{n,r+s}]]\\
    &{}=\delta_{p+1, i}(-1)^s [X_{i-1,n,0}, [X_{j,n-1,0}, X^+_{n,r+s}]]-(-1)^{r+s}\delta_{p+1, j} [X_{i,n-1,0}, [X_{j-1,n,0}, X^+_{n,r+s}]],
\end{align*}
where the second equality is due to \eqref{gat1} and \eqref{Eq903-9.5}, the third equality is due to \eqref{Eq903-11}, and the last equality is due to \eqref{Eq901-10-0}, \eqref{Eq903-14} and \eqref{Eq903-20}.
Thus, we have proved this case.

For the case $j \leq n-2$ and $j - n$ is even, we have
\begin{align*}
    &\phantom{{}={}} [X^+_{p,s}, X_{i,-j,r}]
    = [X_{p,p+1,s}, [X_{i,n-1,0}, [X_{j,n,0}, X^+_{n,r}]]] \\
    &{}= \delta_{p+1,i} [X_{i-1,n-1,s}, [X_{j,n,0}, X^+_{n,r}]]
        + \delta_{p+1,j} [X_{i,n-1,0}, [X_{p,n,s}, X^+_{n,r}]] \\
    &{}= \delta_{p+1,i} [X_{i-1,n-1,0}, [X_{j,n,0}, X^+_{n,r+s}]]+ (-1)^s\delta_{p+1,j} [X_{i,n-1,0},[X_{j-1,n,0}, X^+_{n,r+s}]] \\
    &{}= \delta_{p+1,i} [X_{i-1,n-1,0}, [X_{j,n,0}, X^+_{n,r+s}]]+ (-1)^s\delta_{p+1,j} [X_{i,n,0},[X_{j-1,n-1,0}, X^+_{n,r+s}]] \\
    &{}= \delta_{p+1,i} X_{i-1,-j,r+s}+ (-1)^s\delta_{p+1,j} X_{i,-j+1,r+s},
\end{align*}
where the second equality is due to \eqref{gat1} and \eqref{Eq903-9.5}, the third equality is due to \eqref{Eq903-21} and \eqref{Eq903-10}, and the 4-th equality is due to \eqref{Eq903-1}.

For the case $j = n-1$, we obtain
\begin{align*}
    &\phantom{{}={}} [X^+_{p,s}, X_{i,-j,r}]
    = [X_{p,p+1,s}, [X_{i,n,0}, X^+_{n,r}]] \\
    &{}= \delta_{p+1,i} [X_{i-1,n,s}, X^+_{n,r}]
    = (-1)^s\delta_{p+1,i} X_{i-1, -n+1, r+s}
\end{align*}
by \eqref{gat1}, \eqref{Eq903-9.5} and \eqref{Eq903-10}. 

For the case $j=n$, we have
\begin{align*}
    &\phantom{{}={}} [X^+_{p,s}, X_{i,-j,r}]
    = [X_{p,p+1,s}, [X_{i,n-1,0}, X^+_{n,r}]] \\
    &{}= \delta_{p+1,i} [X_{p,n-1,s}, X^+_{n,r}] = \delta_{p+1,i} X_{i-1,-n,r+s}
\end{align*}
from \eqref{gat1}, \eqref{Eq903-9.5} and \eqref{Eq903-11}.
Thus, the statement is true for every $p \leq n-3$.

Next, we show the case that $p=n-2$. Suppose that $p=n-2$. 
For the case $j \leq n-3$ and $j - n$ is odd, we have 
\begin{align*}
    &\phantom{{}={}} [X^+_{n-2,s}, X_{i,-j,r}]
    = [X^+_{n-2,s}, [X_{i,n,0}, [X_{j,n-1,0}, X^+_{n,r}]]] \\
    &{}= [X_{i,n,0}, [X_{j,n-1,0}, [X^+_{n-2,s}, X^+_{n,r}]]] = 0,
\end{align*}
where the second equality is due to \eqref{gat1} and the assumption that $j\leq n-3$ and the third equality is due to \eqref{Eq903-12}.

For the case $j \leq n-2$ and $j - n$ is even, 
\begin{align*}
    &\phantom{{}={}} [X^+_{n-2,s}, X_{i,-j,r}]
    = [X^+_{n-2,s}, [X_{i,n-1,0}, [X_{j,n,0}, X^+_{n,r}]]] \\
    &{}=  [X_{j,n,0},[X_{i,n-1,0}, [X^+_{n-2,0}, X^+_{n,r}]]] = 0,
\end{align*}
where the second equality is due to \eqref{gat1} and the assumption that $j\leq n-3$, and the last equality is due to \eqref{Eq903-22} and \eqref{gat1}.

For the case $j = n-1$, we obtain 
\begin{align*}
    &\phantom{{}={}} [X^+_{n-2,s}, X_{i,-(n-1),r}]
    = [X_{i,n,0}, [X^+_{n-2,s}, X^+_{n,r}]] \\
    &{}= \delta_{i,n-2} [[X^+_{n-2,0}, X^+_{n-1,0}], [X^+_{n-2,0}, X^+_{n,r+s}]]+\delta_{n-1,i}[X^+_{n-1,0}, [X^+_{n-2,s}, X^+_{n,r}]]\\
    &{}=0-(-1)^{r+s}[[X^+_{n-2,0},X^+_{n-1,0}],X^+_{n,r+s}]
    =-(-1)^{r+s} \delta_{n-1,i} X_{n-2,-n+2,r+s},
\end{align*}
where the second equality is due to \eqref{Eq903-22}, and the third equality is due to\eqref{Eq903-7} and \eqref{Eq903-19}.
For the case $i<n-1$ and $j=n$, by \eqref{Eq903-23} and \eqref{Eq902}, we find that
\begin{align*}
    &\phantom{{}={}} [X^+_{n-2,s}, X_{i,-n,r}]
    = [X_{n-2,n-1,s}, [X_{i,n-1,0}, X^+_{n,r}]]= 0.
\end{align*}
Since the case $(i,j)=(n-1,n)$ is trivial, the statement is true for $p=n-2$.

Finally, we show the case that $p=n-1$. Suppose that $p=n-1$. 
For the case $j \leq n-3$ and $j - n$ is odd, by \eqref{gat1} and \eqref{Eq903-18}, we obtain
\begin{align*}
    &\phantom{{}={}} [X^+_{n-1,s}, X_{i,-j,r}]
    = [X^+_{n-1,s}, [X_{i,n,0}, [X_{j,n-1,0}, X^+_{n,r}]]]\\
    &{}= [X_{i,n,0}, [X^+_{n-1,s}, [X_{j,n-1,0}, X^+_{n, r}]]]\\
    &{}=[X_{i,n,0}, [X_{j,n-1,0}, [X^+_{n-1,s}, X^+_{n, r}]]]-[X_{i,n,0}, [X_{j,n,0}, X^+_{n, r}]]\\
    &{}=[X_{j,n-1,0}, [X_{i,n,0}, [X^+_{n-1,s}, X^+_{n, r}]]]-0\\
    &{}=[X_{j,n-1,0}, [[X_{i,n-3},X_{n-2,n}], [X^+_{n-1,s}, X^+_{n, r}]]]\\
    &{}=[X_{j,n-1,0},[X_{i,n-3}, [X_{n-2,n}, [X^+_{n-1,s}, X^+_{n, r}]]]]=0,
\end{align*}
where the second and 5-th equalities are due to \eqref{gat1}, the third equality is due to the Jacobi identity, the 4-th equality is due to \eqref{Eq903-18} and \eqref{gat1}, the 6-th equality is due to \eqref{gat1} and \eqref{Eq903-9.5}, and the last equality is due to \eqref{Eq903-7}. 

For the case $j \leq n-2$ and $j - n$ is even, 
by \eqref{Eq903-7} and \eqref{Eq903-18}, 
\begin{align*}
    &\phantom{{}={}} [X^+_{n-1,s}, X_{i,-j,r}]
    = [X^+_{n-1,s}, [X_{i,n-1,0}, [X_{j,n,0}, X^+_{n,r}]]] \\
    &{}= -[X_{i,n,s}, [X_{j,n,0}, X^+_{n,r}]] + [X_{i,n-1,0}, [X_{j,n,0}, [X^+_{n-1,s}, X^+_{n,r}]]] \\
    &{}= 0 + [X_{i,n-1,0}, [X_{j,n-2,0}, [[X^+_{n-2,0}, X^+_{n-1,0}], [X^+_{n-1,s}, X^+_{n,r}]]]]
    = 0,
\end{align*}
where the second equality is due to the Jacobi identity, the third equality is due to \eqref{Eq903-18}, and the last equality is due to \eqref{Eq903-7}.

For the case $j = n-1$, we obtain
\begin{align*}
    & [X^+_{n-1,s}, X_{i,-n+1,r}] 
    = [X^+_{n-1,s}, [X_{i,n,0}, X^+_{n,r}]]
    = [X_{i,n,0}, [X^+_{n-1,s}, X^+_{n,r}]] \\
    &{}=(-1)^s[X_{i,n,0},[X^+_{n-1,0},X^+_{n,r+s}]]=(-1)^s[X^+_{n-1,0},[X_{i,n,0},X^+_{n,r+s}]]=0,
\end{align*}
where the second, third equality is due to \eqref{gat1}, the third equality is due to \eqref{Eq901-10}, and the last equality is due to \eqref{Eq903-18}.

For the case $i\leq n-2$ and $j=n$, we have
\begin{align*}
    & [X^+_{n-1,s}, X_{i,-n,r}]
    = [X^+_{n-1,s}, [X_{i,n-2,0}, [X^+_{n-2,0}, X^+_{n,r}]]] \\
    &{}= [X_{i,n-2,0}, [X^+_{n-1,s}, [X^+_{n-2,0}, X^+_{n,r}]]] \\
    &{}=-(-1)^r[X_{i,n-2,0}, [[X^+_{n-2,0},X^+_{n-1,0}],X^+_{n,s+r}]]\\
    &{}=-(-1)^r [X_{i,n,0}, X^+_{n,r+s}] = -(-1)^r X_{i,-n+1,r+s}, 
\end{align*}
where the second equality is due to \eqref{gat1}, the third equality is due to \eqref{Eq903-16}, and the 4-th equality is due to \eqref{gat1} and \eqref{Eq903-9.5}.

For the case that $(i,j)=(n-1,n)$, by the definition of $X_{n-1,-n+1,r}$, \eqref{Eq901-10} and \eqref{Eq901-17-1}, we have
\begin{align*}
    & [X^+_{n-1,s}, X_{n-1,-n,r}] 
    = [X^+_{n-1,s}, X^+_{n,r}] = (-1)^s \delta_{r+s,odd} X_{n-1,-n+1,r+s}.
\end{align*}
Thus, the statement is true for $p = n-1$. 

\subsection{Proof for $[X^+_{n,s}, X_{i,-j,r}]$}
We show the statement for each of the cases $j\leq n-2$, $j=n-1$ and $j=n$.

For the case $j \leq n-2$, by \eqref{Eq903-9.5} and the Jacobi identity, we have 
\begin{align*}
    &\phantom{{}={}} [X^+_{n,s}, X_{i,-j,r}]
    = \left( \prod_{u=i}^{j-2} \ad (X^+_{u,0}) \right)
        \left( \prod_{u=j+1}^n \ad (X_{u-2,u,0}) \right) (X^+_{n,r}) \\
    &{}= \left( \prod_{u=i}^{j-2} \ad (X^+_{u,0}) \right)
        \left( \prod_{u=j+1}^{n-2} \ad (X_{u-2,u,0}) \right) \\
        &\phantom{{}={}} \quad\quad\quad [X^+_{n,s}, [[X^+_{n-3,0}, X^+_{n-2,0}], [[X^+_{n-2,0}, X^+_{n-1,0}], X^+_{n,r}]]]\\
    &{}=[[X^+_{n,s}, [X^+_{n-2,0}, X^+_{n-3,0}]], [[X^+_{n-1,0}, X^+_{n-2,0}], X^+_{n,r}]]\\
     &\phantom{{}={}} +[[X^+_{n-2,0}, X^+_{n-3,0}], [X^+_{n,s},[[X^+_{n-1,0}, X^+_{n-2,0}], X^+_{n,r}]]].
\end{align*}
The first term of the right hand side is equal to zero by \eqref{Eq903-4}. We also find that the second term of the right hand side is equal to zero since we obtain
\begin{align}
&\quad[X^+_{n,s},[[X^+_{n-1,0}, X^+_{n-2,0}], X^+_{n,r}]]\nonumber\\
&=[[X^+_{n,s},X^+_{n-1,0}],[X^+_{n-2,0},X^+_{n,r}]]+[[X^+_{n-1,0},X^+_{n,r}],[X^+_{n,s},X^+_{n-2,0}]]=0+0,\label{first5}
\end{align}
where the first equality is due to the Jacobi identity, \eqref{Eq901-17-1}, \eqref{Eq901-17-2} and \eqref{Eq902}, and the second equality is due to \eqref{Eq903-3}, \eqref{Eq901-10-0} and \eqref{Eq901-10}.

For the case $j = n-1$, the following relation holds:
\begin{align*}
    &\phantom{{}={}} [X^+_{n,s}, X_{i,-j,r}]
    = \left[ X^+_{n,s}, \left( \prod_{u=i}^{n-3} \ad (X^+_{u,0}) \right) 
        [[X^+_{n-1,0}, X^+_{n-2,0}], X^+_{n,r}] \right] \\
    &{}= \left( \prod_{u=i}^{n-3} \ad (X^+_{u,0}) \right) 
        [X^+_{n,s}, [[X^+_{n-1,0}, X^+_{n-2,0}], X^+_{n,r}]] = 0,
\end{align*}
where the second equality is due to \eqref{gat1} and \eqref{Eq903-9.5}, and the last equality is due to \eqref{first5}.

Suppose that $j=n$. In the case that $i=n-1$, this directly follows from \eqref{Eq901-10-2}.
For $i\leq n-2$, by \eqref{Eq903-9.5} and \eqref{Eq902}, we have 
\begin{align*}
    &\phantom{{}={}} [X^+_{n,s}, X_{i,-j,r}]
    = \left[ X^+_{n,s}, \left( \prod_{u=i}^{n-3} \ad (X^+_{u,0}) \right) 
        ([X^+_{n-2,0}, X^+_{n,r}]) \right] \\
    &{}= \left( \prod_{u=i}^{n-3} \ad (X^+_{u,0}) \right) [X^+_{n,s}, [X^+_{n-2,0}, X^+_{n,r}]] = 0.
\end{align*}

\subsection{Proof for $[X^+_{p,s}, X_{i,-i,2r+1}]$}
We show the statement for each of the cases $i\leq n-2$ and $i=n-1$.

First, we show the case $i\leq n-2$. Suppose $i \leq n-2$. 
From the Jacobi identity, we have 
\begin{align}
    &\phantom{{}={}} [X^+_{p,s}, X_{i,-i,2r+1}]
    = [X_{p,p+1,s}, [X_{i,n,0}, [X_{i,n-1,0}, X^+_{n,2r+1}]]]\nonumber \\
    &{}= [[X_{p,p+1,s}, X_{i,n,0}], [X_{i,n-1,0}, X^+_{n,2r+1}]]
    + [X_{i,n,0}, [[X_{p,p+1,s}, X_{i,n-1,0}], X^+_{n,2r+1}]] \nonumber\\
    &{} \quad \quad \quad  + [X_{i,n,0}, [X_{i,n-1,0}, [X^+_{p,s}, X^+_{n,2r+1}]]]. \label{first6}
\end{align}
By \eqref{gat1} and \eqref{Eq901-17}, we have
\begin{align*}
\eqref{first6}_1&=\delta_{p+1,i} [X_{p,n,s}, [X_{i,n-1,0}, X^+_{n,2r+1}]],\\
\eqref{first6}_2&=\delta_{p+1,i} [X_{i,n,0}, [X_{p,n-1,s}, [X^+_{n,2r+1}]]]
            - \delta_{p,n-1} [X_{i,n,0}, [X_{i,p+1,s}, X^+_{n,2r+1}]],\\
\eqref{first6}_3&=(\delta_{p,n-1} + \delta_{p,n-2}) [X_{i,n,0}, [X_{i,n-1,0}, [X^+_{p,s}, X^+_{n,2r+1}]]].
\end{align*}
Thus, it is enough to compute the following three terms:
\begin{gather}
[X_{i-1,n,s}, [X_{i,n-1,0}, X^+_{n,2r+1}]]
        + [X_{i,n-1,0}, [X_{i-1,n,s}, X^+_{n,2r+1}]],\label{first7}\\
[X_{i,n,0}, [X_{i,n-1,0}, [X^+_{n-2,s}, X^+_{n,2r+1}]]],\label{first8}\\
[X_{i,n,0}, [X_{i,n-1,0}, [X^+_{n-1,s}, X^+_{n,2r+1}]]]
        - [X_{i,n,0}, [X_{i,n,s}, X^+_{n,2r+1}]].\label{first9}
\end{gather}
By \eqref{Eq903-12}, we find that \eqref{first8} is equal to zero.

From the assumption that $i\leq n-2$, the Jacobi identity and \eqref{gat1}, we find that
\begin{align*}
    \eqref{first7}
    &{}= 2[X_{i,n-1,0}, [X_{i,n,s}, X^+_{n,2r+1}]]
    = (-1)^s 2 [X_{i,n-1,0}, [X_{i-1,n,0}, X^+_{n,2r+1+s}]]\\
    &{}=-(-1)^s (-1)^{2r+1+s} 2[X_{i,n,0}, [X_{i-1,n-1,0}, X^+_{n,2r+1+s}]]] \\
    &{}= 2[X_{i,n,0}, [X_{i-1,n-1,0}, X^+_{n,2r+1+s}]]].
\end{align*}
Here we used \eqref{Eq903-10} in the second equality and \eqref{Eq903-14} in the third equality.
Hence, in either case $n - i$ is even or odd, we have 
\begin{align*}
    \eqref{first7}
    &{}= 
    2 ((-1)^s \delta_{n-i,odd} + \delta_{n-i,even}) X_{i-1,-i,2r+1+s}. 
\end{align*}

Since $[X_{i,n-1,0}, X^+_{n-1,s}] = X_{i,n,s}$ from \eqref{gat1}, we have
\begin{align*}
 \eqref{first9}
    &{}= (-1)^s([X_{i,n,0}, [X_{i,n-1,0}, [X^+_{n-1,0}, X^+_{n,2r+1+s}]]]
        - [X_{i,n,0}, [X_{i,n,0}, X^+_{n,2r+1+s}]]) \\
    &{}= (-1)^s([X_{i,n,0}, [[X_{i,n-1,0}, X^+_{n-1,0}], X^+_{n,2r+1+s}]]  \\
    & \quad \quad+ [X_{i,n,0}, [X^+_{n-1,0}, [X_{i,n-1,0}, X^+_{n,2r+1+s}]]] - [X_{i,n,0}, [X_{i,n,0}, X^+_{n,2r+s+1}]] ) \\
    &{}= (-1)^s [X_{i,n,0}, [X^+_{n-1,0}, [X_{i,n-1,0}, X^+_{n,2r+1+s}]]],
\end{align*}
where the first equality is due to \eqref{Eq903-10}.
We have the relations $X_{i,n-1,0} = [X_{i,n-2,0}, X^+_{n-2,0}]$,
$[X^+_{n-1,0}, X_{i,n-2,0}] = 0$ and
$[X_{i,n-2,0}, X^+_{n,2r+1+s}] = 0$ from \eqref{gat1} and \eqref{Eq903-9.5}, hence we obtain
\begin{align*}
    \eqref{first9}
    &{}= (-1)^s [X_{i,n,0}, [X^+_{n-1,0}, [X_{i,n-2,0}, [X^+_{n-2,0}, X^+_{n,2r+1+s}]]]] \\
    &{}= (-1)^s [X_{i,n,0}, [X_{i,n-2,0}, [X^+_{n-1,0}, [X^+_{n-2,0}, X^+_{n,2r+1+s}]]]] \\
    &{}= (-1)^s (-1)^{2r+1+s} [X_{i,n,0}, [X_{i,n-2,0}, [X^+_{n,2r+1+s}, [X^+_{n-2,0}, X^+_{n-1,0}]]]] \\
    &{}= [X_{i,n,0}, [X_{i,n-2,0}, [[X^+_{n-2,0}, X^+_{n-1,0}], X^+_{n,2r+1+s}]]] \\
    &{}= [X_{i,n,0}, [[X_{i,n-2,0}, [X^+_{n-2,0}, X^+_{n-1,0}]], X^+_{n,2r+1+s}]] \\
    &{}= [X_{i,n,0}, [X_{i,n,0}, X^+_{n,2r+1+s}]] = 0.
\end{align*}
Here we used \eqref{Eq903-1} in the third equality, and \eqref{Eq903-18-2} in the last equality. 

Next, we show the case $i=n-1$. For the case $i = n-1$, we obtain
\begin{align*}
    &\phantom{{}={}} [X^+_{p,s}, X_{n-1,-n+1,2r+1}] \\
    &{}= [[X_{p,p+1,s}, X_{n-1,n,0}], X^+_{n,2r+1}]
        + [X^+_{n-1,0}, [X^+_{p,s}, X^+_{n,2r+1}]] \\
    &{}= \delta_{p+1,n-1} [X_{p,n,s}, X^+_{n,2r+1}]
        + \delta_{p,n-2} [X^+_{n-1,0}, [X^+_{p,s}, X^+_{n,2r+1}]] \\
    &{}= \delta_{p+1,n-1} ([X_{n-2,n,s}, X^+_{n,2r+1}]
        + [X^+_{n-1,0}, [X^+_{n-2,s}, X^+_{n,2r+1}]]), 
\end{align*}
where the second equality is due to \eqref{gat1}, \eqref{Eq901-17-1} and \eqref{Eq901-17-2}.
Since $X_{n-2,n,s} = [X^+_{n-2,0}, X^+_{n-1,s}]$ holds by \eqref{gat1}, we have 
\begin{align*}
    &\phantom{{}={}} \delta_{p+1,n-1} ([X_{n-2,n,s}, X^+_{n,2r+1}]
        + [X^+_{n-1,0}, [X^+_{n-2,s}, X^+_{n,2r+1}]]) \\
    &{}= \delta_{p+1,n-1} ([[X^+_{n-2,0}, X^+_{n-1,s}], X^+_{n,2r+1}]
        + [X^+_{n-1,0}, [X^+_{n-2,s}, X^+_{n,2r+1}]]) \\
    &{}= \delta_{p+1,n-1} ((-1)^s [[X^+_{n-2,0}, X^+_{n-1,0}], X^+_{n,2r+1+s}] 
        + (-1)^{2r+1+s} [X^+_{n,2r+1+s}, [X^+_{n-2,0}, X^+_{n-1,0}]]) \\
    &{}= 2 (-1)^s \delta_{p+1,n-1} [[X^+_{n-2,0}, X^+_{n-1,0}], X^+_{n,2r+1+s}] 
    = 2 (-1)^s \delta_{p+1,n-1} X_{n-2,-n+1,2r+1+s}
\end{align*}
using \eqref{Eq903-16} and \eqref{Eq903-19} in the second equality.

\subsection{Proof for $[X^+_{n,s}, X_{i,-i,2r+1}]$}
We show the statement for each of the cases $i\leq n-3$, $i=n-2$ and $i=n-1$.

For $i \leq n-3$, the following relation holds:
\begin{align*}
    &\phantom{{}={}} [X^+_{n,s}, X_{i,-i,2r+1}] \\
    &{}= \Bigg[ X^+_{r,s}, \ad (X^+_{i,0}) 
        \left( \prod_{u=i+2}^{n-2} \ad (X_{u-2,u,0}) \right)  \\
       &\phantom{{}={}} \quad\quad\quad \ad ([X^+_{n-3,0}, X^+_{n-2,0}]) \ad ([X^+_{n-2,0}, X^+_{n-1,0}]) (X^+_{n,2r+1}) \Bigg] \\
    &{}= \ad (X^+_{i,0}) 
        \left( \prod_{u=i+2}^{n-2} \ad (X^+_{u-2,u,0}) \right) \\
        &\phantom{{}={}} \quad\quad\quad [X^+_{n,s}, [[X^+_{n-2,0}, X^+_{n-3,0}], [[X^+_{n-1,0}, X^+_{n-2,0}], X^+_{n,2r+1}]]]\\
     &{}= \ad (X^+_{i,0}) 
        \left( \prod_{u=i+2}^{n-2} \ad ([X^+_{u-1,0}, X^+_{u-2,0}]) \right) \\
        &\phantom{{}={}} \quad\quad\quad [[X^+_{n,s}, [X^+_{n-2,0}, X^+_{n-3,0}]], [[X^+_{n-1,0}, X^+_{n-2,0}], X^+_{n,2r+1}]] 
        = 0,
\end{align*}
where the second equality is due to \eqref{Eq903-9.5} and \eqref{gat1},  the third equality is due to \eqref{first5}, and the last equality is due to \eqref{Eq903-4}.

For $i = n-2$, we obtain
\begin{align*}
&\quad[X^+_{n,s}, X_{n-2,-n+2,2r+1}]= [X^+_{n,s}, [X^+_{n-2,0}, [[X^+_{n-2,0}, X^+_{n-1,0}], X^+_{n,2r+1}]]]\\
&=[X^+_{n,s},  [[X^+_{n-2,0}, X^+_{n-1,0}], [X^+_{n-2,0},X^+_{n,2r+1}]]]\\
&=[[X^+_{n,s},[X^+_{n-2,0}, X^+_{n-1,0}]], [X^+_{n-2,0},X^+_{n,2r+1}]]=0.
\end{align*}
where the second and third equalities are due to \eqref{Eq902} and the last equality is due to \eqref{Eq903-5}.

For $i=n-1$, by \eqref{Eq902} and \eqref{Eq901-14}, we have
\begin{align*}
    [X^+_{n,s}, X_{n-1,-n+1,2r+1}] &{}= [X^+_{n,s}, [X^+_{n-1,0}, X^+_{n,2r+1}]]= 0.
\end{align*}

\section{Proof of Theorem~\ref{Main}}
It is enough to show the compatibility with the relations \eqref{q901}-\eqref{q903-7}. By \eqref{SS1}, the compatibilities with the relations \eqref{q901-1-1}, \eqref{q901-2--1}, \eqref{q901-17} and \eqref{q902}  are trivial. Thus, we need to show the compatibility with \eqref{q901}, \eqref{q901-1}, \eqref{q901-2-0}, \eqref{q901-2-1}, \eqref{q901-3}, \eqref{q901-10}, \eqref{q903-1}, \eqref{q903-3} and \eqref{q903-7}. We will give the proof of the compatibility with \eqref{q901}, \eqref{q901-1}, \eqref{q901-2-1}, \eqref{q901-3}, \eqref{q901-10}, \eqref{q903-1}, \eqref{q903-3} and \eqref{q903-7}. The other compatibilities can be proven by the similar way.

\subsection{Compatibility with \eqref{q901}}
Let us set 
\begin{align*}
V_i&=\sum_{1\leq a\leq i}S^{(1)}_{i,a}S^{(1)}_{a,i},\ W_i=\sum_{1\leq b\leq i}S^{(1)}_{i,-b}S^{(1)}_{-b,i}.
\end{align*}
By \eqref{SS1} and the relation $S^{(1)}_{i,-i}=0$, we can rewrite
\begin{align*}
&\quad[\Phi(H_{i,1}),\Phi(H_{j,1})]\\
&=[-S^{(2)}_{i,i}+S^{(2)}_{i+1,i+1}+\hbar(V_i+W_i-V_{i+1}-W_{i+1}),\\
&\qquad\qquad-S^{(2)}_{j,j}+S^{(2)}_{j+1,j+1}+\hbar(V_j+W_j-V_{j+1}-W_{j+1})].
\end{align*}
Thus, for the proof of the compatibility with \eqref{q901}, it is enough to show the relation
\begin{align*}
[-S^{(2)}_{i,i}+\hbar(V_i+W_i),-S^{(2)}_{j,j}+\hbar(V_j+W_j)]&=0.
\end{align*}
By \eqref{SS1}, we obtain
\begin{align}
[S^{(2)}_{i,i},V_j+W_j]&=\delta(i\leq j)(S^{(1)}_{j,i}S^{(2)}_{i,j}-S^{(2)}_{j,i}S^{(1)}_{i,j})+\delta(i\leq j)(S^{(1)}_{i,-j}S^{(2)}_{-j,i}-S^{(2)}_{i,-j}S^{(1)}_{-j,i}).\label{SVW}
\end{align}
In the case that $i<j$, by \eqref{SVW}, we have
\begin{align}
[V_i+W_i,S^{(2)}_{j,j}]=0.\label{SVW2}
\end{align}
By \eqref{SVW}, \eqref{SVW2} and \eqref{SS3}, we have
\begin{align}
[S^{(2)}_{i,i},S^{(2)}_{j,j}]-[S^{(2)}_{i,i},\hbar(V_j+W_j)]-[\hbar(V_i+W_i),S^{(2)}_{j,j}]=0\label{SVW3}
\end{align}
for $i<j$. By the symmetry with respect to $i$ and $j$, \eqref{SVW3} holds for any $1\leq i,j\leq n$. Thus, we need to show that $[V_i+W_i,V_j+W_j]=0$. 

By \eqref{SS1} and a direct computation, for $i\neq j$, we obtain
\begin{align}
[V_i,V_j]&=0,\\
[V_i,W_j]&=-\delta(i>j)\sum_{1\leq b\leq j}S^{(1)}_{i,j}S^{(1)}_{j,-b}S^{(1)}_{-b,i}+\delta(i>j)\sum_{1\leq b\leq j}S^{(1)}_{i,b}S^{(1)}_{j,-b}S^{(1)}_{-j,i}\nonumber\\
&\quad+\delta(i>j)\sum_{1\leq b\leq j}S^{(1)}_{i,-b}S^{(1)}_{-b,j}S^{(1)}_{j,i}-\delta(i>j)\sum_{1\leq b\leq j}S^{(1)}_{i,-j}S^{(1)}_{-b,j}S^{(1)}_{b,i},\label{1}\\
[W_i,W_j]&=-\delta(i<j)\sum_{1\leq a\leq i}S^{(1)}_{i,-a}(S^{(1)}_{j,i})S^{(1)}_{-a,j}-\delta(i>j)\sum_{1\leq b\leq j}S^{(1)}_{i,-b}S^{(1)}_{j,i}S^{(1)}_{-b,j}\nonumber\\
&\quad-\delta(i<j)\sum_{1\leq a\leq i}S^{(1)}_{i,-a}S^{(1)}_{-a,-j}S^{(1)}_{-i,j}+\delta(i>j)\sum_{1\leq b\leq j}S^{(1)}_{i,-j}S^{(1)}_{b,i}S^{(1)}_{-b,j}\nonumber\\
&\quad+\delta(i<j)\sum_{1\leq a\leq i}S^{(1)}_{j,-a}S^{(1)}_{i,j}S^{(1)}_{-a,i}+\delta(i>j)\sum_{1\leq b\leq j}S^{(1)}_{j,-b}S^{(1)}_{i,j}S^{(1)}_{-b,i}\nonumber\\
&\quad-\delta(i>j)\sum_{1\leq b\leq j}S^{(1)}_{j,-b}S^{(1)}_{i,b}S^{(1)}_{-j,i}+\delta(i<j)\sum_{1\leq a\leq i}S^{(1)}_{j,-i}S^{(1)}_{-j,-a}S^{(1)}_{-a,i}.\label{2}
\end{align}
For the convenience of the notation, we denote the $x$-th term of the right hand side of the equality $(\cdot)$ in the case that $i=a,j=b$ by $(\cdot)_{a,b,x}$.
We assume that $i<j$. We can divide $[V_i+W_i,V_j+W_j]$ into two pieces as follows:
\begin{align*}
&\quad-\eqref{1}_{j,i,1}-\eqref{1}_{j,i,3}+\eqref{2}_{i,j,1}+\eqref{2}_{i,j,5}\\
&=\sum_{1\leq b\leq i}S^{(1)}_{j,i}S^{(1)}_{i,-b}S^{(1)}_{-b,j}-\sum_{1\leq b\leq i}S^{(1)}_{j,-b}S^{(1)}_{-b,i}S^{(1)}_{i,j}\\
&\quad-\sum_{1\leq a\leq i}S^{(1)}_{i,-a}S^{(1)}_{j,i}S^{(1)}_{-a,j}+\sum_{1\leq a\leq i}S^{(1)}_{j,-a}S^{(1)}_{i,j}S^{(1)}_{-a,i}=0,\\
&\quad-\eqref{1}_{j,i,2}-\eqref{1}_{j,i,4}+\eqref{2}_{i,j,3}+\eqref{2}_{i,j,8}\\
&=-\sum_{1\leq b\leq i}S^{(1)}_{j,b}S^{(1)}_{i,-b}S^{(1)}_{-i,j}+\sum_{1\leq b\leq i}S^{(1)}_{j,-i}S^{(1)}_{-b,i}S^{(1)}_{b,j}\\
&\quad-\sum_{1\leq a\leq i}S^{(1)}_{i,-a}S^{(1)}_{-a,-j}S^{(1)}_{-i,j}+\sum_{1\leq a\leq i}S^{(1)}_{j,-i}S^{(1)}_{-j,-a}S^{(1)}_{-a,i}=0.
\end{align*}
\subsection{Compatibility with \eqref{q901-1}}
We only show the case $j=n$. The other cases can be proven in a similar way.
By \eqref{SS1} and \eqref{SS5}, we obtain
\begin{align}
&\quad[-(S^{(2)}_{i,i}-S^{(2)}_{i+1,i+1}),S^{(1)}_{n-1,-n}]\nonumber\\
&=-\delta_{i,n-1}S^{(2)}_{n-1,-n}+\delta_{i,n-2}S^{(2)}_{n-1,-n}-\delta_{i,n-1} S^{(2)}_{n,-n+1}\nonumber\\
&=-2\delta_{i,n-1}S^{(2)}_{n-1,-n}+\delta_{i,n-2}S^{(2)}_{n-1,-n}-\delta_{i,n-1}\hbar S^{(1)}_{n-1,-n}.\label{0000}
\end{align}
By \eqref{SS1}, we have
\begin{align}
&\quad[-\dfrac{i+n-1}{2}\hbar (S^{(1)}_{i,i}-S^{(1)}_{i+1,i+1})-\dfrac{\hbar}{2}(S^{(1)}_{i,i})^2-\dfrac{\hbar}{2}(S^{(1)}_{i+1,i+1})^2,S^{(1)}_{n-1,-n}]\nonumber\\
&=\dfrac{2n-3}{2}\delta_{i,n-2}S^{(1)}_{n-1,-n}-\hbar\delta_{i,n-1}S^{(1)}_{n-1,-n}S^{(1)}_{n-1,n-1}-\dfrac{\hbar}{2}\delta_{i,n-1}S^{(1)}_{n-1,-n}\nonumber\\
&\quad-\delta_{i,n-2}\hbar S^{(1)}_{n-1,n-1}S^{(1)}_{n-1,-n}+\delta_{i,n-2}\dfrac{\hbar}{2}S^{(1)}_{n-1,-n}+\delta_{i,n-1}{\hbar}S^{(1)}_{n-1,-n}S^{(1)}_{n,n}-\delta_{i,n-1}\dfrac{\hbar}{2}S^{(1)}_{n,-n+1},\label{1111}\\
&\quad[\hbar \sum_{1\leq a\leq i}S^{(1)}_{i,a}S^{(1)}_{a,i},S^{(1)}_{n-1,-n}]-[\hbar \sum_{1\leq a\leq i}S^{(1)}_{i+1,a}S^{(1)}_{a,i+1},S^{(1)}_{n-1,-n}]\nonumber\\
&=\delta_{i,n-1}\hbar \sum_{1\leq a\leq n-1}S^{(1)}_{n-1,a}S^{(1)}_{a,-n}+\delta_{i,n-1}\hbar S^{(1)}_{n-1,-n}S^{(1)}_{n-1,n-1}-\delta_{i+1,n-1}\hbar \sum_{1\leq a\leq n-2}S^{(1)}_{n-1,a}S^{(1)}_{a,-n}\nonumber\\
&\quad+\delta_{i,n-1}\hbar \sum_{1\leq a\leq n-1}S^{(1)}_{n,a}S^{(1)}_{a,-n+1},\label{2222}\\
&\quad[\hbar \sum_{1\leq a\leq i}S^{(1)}_{i,-a}S^{(1)}_{-a,i},S^{(1)}_{n-1,-n}]-[\hbar \sum_{1\leq a\leq i}S^{(1)}_{i+1,-a}S^{(1)}_{-a,i+1},S^{(1)}_{n-1,-n}]\nonumber\\
&=\delta_{i,n-1}\hbar \sum_{1\leq a\leq n-1}S^{(1)}_{n-1,-a}S^{(1)}_{-a,-n}-\delta_{i+1,n-1}\hbar \sum_{1\leq a\leq n-2}S^{(1)}_{n-1,-a}S^{(1)}_{-a,-n}\nonumber\\
&\quad+\delta_{i,n-1}\hbar \sum_{1\leq a\leq n-1}S^{(1)}_{n,-a}S^{(1)}_{-a,-n+1}+\delta_{i,n-1}\hbar S^{(1)}_{n,-n+1}S^{(1)}_{n,n}.\label{3333}
\end{align}
By a direct computation, we obtain
\begin{align*}
&\quad\eqref{0000}_2+\eqref{1111}_{1}+\eqref{1111}_{4}+\eqref{1111}_{5}+\eqref{2222}_{3}+\eqref{3333}_4\\
&=\delta_{i,n-2}S^{(2)}_{n-1,-n}+(n-1)\delta_{i,n-2}\hbar S^{(1)}_{n-1,-n}-\delta_{i,n-2}\hbar \sum_{1\leq a\leq n-2}S^{(1)}_{n-1,a}S^{(1)}_{a,-n}\\
&\quad-\delta_{i,n-2}\hbar \sum_{1\leq a\leq n-2}S^{(1)}_{n-1,-a}S^{(1)}_{-a,-n}\\
&=-\delta_{i,n-2}\Phi(X^+_{n,1})
\end{align*}
and
\begin{align*}
&\quad\eqref{0000}_1+\eqref{0000}_3+\eqref{1111}_{2}+\eqref{1111}_{3}+\eqref{1111}_{6}+\eqref{1111}_{7}\\
&\qquad\qquad+\eqref{2222}_{1}+\eqref{2222}_{2}+\eqref{2222}_{4}+\eqref{3333}_1+\eqref{3333}_3+\eqref{3333}_4\\
&=-2\delta_{i,n-1}S^{(2)}_{n-1,-n}+2\delta_{i,n-1}\hbar \sum_{1\leq a\leq n-1}S^{(1)}_{n-1,a}S^{(1)}_{a,-n}\\
&\qquad\qquad+2\delta_{i+1,n}\hbar \sum_{1\leq a\leq n-1}S^{(1)}_{n-1,-a}S^{(1)}_{-a,-n}-2(n-1)\hbar S^{(1)}_{n-1,-n}\\
&=2\delta_{i,n-1}\Phi(X^+_{n,1}).
\end{align*}
Thus, we obtain
\begin{align*}
[\Phi(H_{i,1})-\dfrac{\hbar}{2}\Phi(H_{i,0}^2),X^+_{n,0}]&=(2\delta_{i,n-1}-\delta_{i,n-2})\Phi(X^+_{n,1}).
\end{align*}
We can prove the $-$ case in the same way.
\subsection{Compatibility with \eqref{q901-3}}
By \eqref{SS1} and \eqref{SS6}, we obtain
\begin{align*}
[-S^{(2)}_{n-1,-n},S^{(1)}_{-n,n-1}]&=-S^{(2)}_{n-1,n-1}+S^{(2)}_{-n,-n}=-S^{(2)}_{n-1,n-1}+S^{(2)}_{n,n}+\hbar S^{(1)}_{n,n}.
\end{align*}
By \eqref{SS1} and \eqref{SS1}, we have
\begin{align*}
&\quad-(n-1)\hbar[S^{(1)}_{n-1,-n},S^{(1)}_{-n,n-1}]+[\hbar\sum_{1\leq a\leq n-1}S^{(1)}_{n-1,a}S^{(1)}_{a,-n},S^{(1)}_{-n,n-1}]\\
&\qquad\qquad+[\hbar\sum_{1\leq a\leq n-1}S^{(1)}_{n-1,-a}S^{(1)}_{-a,-n},S^{(1)}_{-n,n-1}]\\
&=-(n-1)\hbar(S^{(1)}_{n-1,n-1}+S^{(1)}_{n,n})+\hbar S^{(1)}_{n-1,n-1}S^{(1)}_{n,n}\\
&\quad+\hbar\sum_{1\leq a\leq n-1}S^{(1)}_{n-1,a}S^{(1)}_{a,n-1}-\hbar\sum_{1\leq a\leq n-1}S^{(1)}_{-n,a}S^{(1)}_{a,-n}\\
&\quad+\hbar\sum_{1\leq a\leq n-2}S^{(1)}_{n-1,-a}S^{(1)}_{-a,n-1}-\hbar\sum_{1\leq a\leq n-2}S^{(1)}_{-n,-a}S^{(1)}_{-a,-n}.
\end{align*}
Since 
\begin{align*}
&\quad-\hbar\sum_{1\leq a\leq n-1}S^{(1)}_{-n,a}S^{(1)}_{a,-n}-\hbar\sum_{1\leq a\leq n-2}S^{(1)}_{-n,-a}S^{(1)}_{-a,-n}\\
&=-\hbar\sum_{1\leq a\leq n-1}S^{(1)}_{n,-a}S^{(1)}_{-a,n}-\hbar\sum_{1\leq a\leq n-2}S^{(1)}_{n,a}S^{(1)}_{a,n}-(2n-3)\hbar S^{(1)}_{-n,-n}
\end{align*}
holds by a direct computation, we have proved the compatibility with \eqref{q901-3}.
\subsection{Compatibility with \eqref{q901-10}}
We only show the case that $i=n$. The other cases can be proven in a similar way.
By \eqref{SS1}, we obtain
\begin{align*}
&\quad[-S^{(2)}_{n-1,-n}-(n-1)\hbar S^{(1)}_{n-1,-n}+\hbar\sum_{1\leq a\leq n-1}S^{(1)}_{n-1,a}S^{(1)}_{a,-n}+\hbar\sum_{1\leq a\leq n-1}S^{(1)}_{n-1,-a}S^{(1)}_{-a,-n},S^{(1)}_{n-1,-n}]\\
&=0+\hbar\sum_{1\leq a\leq n-1}(\delta_{a,n-1}S^{(1)}_{n-1,-n})S^{(1)}_{a,-n}+0+0\\
&=\hbar S^{(1)}_{n-1,-n}S^{(1)}_{n-1,-n}.
\end{align*}
Thus, we have proved the $+$ case. The $-$ case can be proven in a similar way.

\subsection{Compatibility with \eqref{q901-2-1}}
First, we compute $[\Phi(X^+_{n-1,1}),\Phi(X^-_{n-1,1})]$. By \eqref{SS1} and \eqref{SS6.5}, we obtain
\begin{align}
&\quad[S^{(2)}_{n-1,n},S^{(2)}_{n,n-1}]\nonumber\\
&=S^{(3)}_{n-1,n-1}+\hbar S^{(1)}_{n,n}S^{(2)}_{n-1,n-1}-S^{(3)}_{n,n}-\hbar S^{(2)}_{n,n}S^{(1)}_{n-1,n-1}-\hbar S^{(1)}_{n-1,-n}S^{(2)}_{-n,n-1}+\hbar S^{(2)}_{n-1,-n}S^{(1)}_{-n+1,n},\label{eee1}\\
&\quad[S^{(2)}_{n-1,n},\hbar (n-1) S^{(1)}_{n,n-1}]+[\hbar (n-1) S^{(1)}_{n-1,n},S^{(2)}_{n,n-1}]+[\hbar (n-1) S^{(1)}_{n-1,n},\hbar (n-1) S^{(1)}_{n,n-1}]\nonumber\\
&=2(n-1)\hbar (S^{(2)}_{n-1,n-1}-S^{(2)}_{n,n})+\hbar^2 (n-1)^2(S^{(1)}_{n-1,n-1}-S^{(1)}_{n,n}),\label{eee2}\\
&\quad[S^{(2)}_{n-1,n},\hbar \sum_{a\in I_{n-1}}S^{(1)}_{n,a}S^{(1)}_{a,n-1}]\nonumber\\
&=\hbar \sum_{a\in I_{n-1}}S^{(2)}_{n-1,a}S^{(1)}_{a,n-1}-\hbar \sum_{a\in I_{n-1}}S^{(1)}_{n,a}S^{(2)}_{a,n}-\hbar S^{(2)}_{n,n}S^{(1)}_{n-1,n-1}+\hbar S^{(1)}_{n,-n+1}S^{(2)}_{-n+1,n},\label{eee3}\\
&\quad[\hbar S^{(1)}_{n-1,n},\hbar \sum_{a\in I_{n-1}}S^{(1)}_{n,a}S^{(1)}_{a,n-1}]\nonumber\\
&=\hbar^2 \sum_{a\in I_{n-1}}S^{(1)}_{n-1,a}S^{(1)}_{a,n-1}-\hbar^2\sum_{a\in I_{n-1}}S^{(1)}_{n,a}S^{(1)}_{a,n}-\hbar^2 S^{(1)}_{n,n}S^{(1)}_{n-1,n-1}+\hbar^2 S^{(1)}_{n,-n+1}S^{(1)}_{-n+1,n},\label{eee4}\\
&\quad[\hbar \sum_{a\in I_{n-1}}S^{(1)}_{n-1,a}S^{(1)}_{a,n},S^{(2)}_{n,n-1}]\nonumber\\
&=\hbar \sum_{a\in I_{n-1}}S^{(1)}_{n-1,a}S^{(2)}_{a,n-1}-\hbar \sum_{a\in I_{n-1}}S^{(2)}_{n,a}S^{(1)}_{a,n}-\hbar S^{(1)}_{n-1,n-1}S^{(2)}_{n,n}+\hbar S^{(2)}_{n,-n+1}S^{(1)}_{-n+1,n},\label{eee5}\\
&\quad[\hbar \sum_{a\in I_{n-1}}S^{(1)}_{n-1,a}S^{(1)}_{a,n},\hbar S^{(1)}_{n,n-1}]\nonumber\\
&=\hbar^2 \sum_{a\in I_{n-1}}S^{(1)}_{n-1,a}S^{(1)}_{a,n-1}-\hbar^2 \sum_{a\in I_{n-1}}S^{(1)}_{n,a}S^{(1)}_{a,n}-\hbar^2 S^{(1)}_{n-1,n-1}S^{(1)}_{n,n}+\hbar^2 S^{(1)}_{n,-n+1}S^{(1)}_{-n+1,n}.\label{eee6}\\
&\quad[\hbar \sum_{a\in I_{n-1}}S^{(1)}_{n-1,a}S^{(1)}_{a,n},\hbar \sum_{b\in I_{n-1}}S^{(1)}_{n,b}S^{(1)}_{b,n-1}]\nonumber\\
&=-\hbar^2 \sum_{b\in I_{n-1}}S^{(1)}_{n-1,n-1}S^{(1)}_{n,b}S^{(1)}_{b,n}+\hbar^2 S^{(1)}_{n,-n+1}S^{(1)}_{-n,n-1}+2(n-1)\hbar^2 S^{(1)}_{n-1,-n}S^{(1)}_{-n,n-1}\nonumber\\
&\quad+\hbar^2 \sum_{a,b\in I_{n-1}}S^{(1)}_{n-1,a}S^{(1)}_{a,b}S^{(1)}_{b,n-1}-\hbar^2 \sum_{a\in I_{n-1}}S^{(1)}_{n-1,a}S^{(1)}_{n,n}S^{(1)}_{b,n-1}-\hbar^2 \sum_{a,b\in I_{n-1}}S^{(1)}_{n,b}S^{(1)}_{b,a}S^{(1)}_{a,n}.\label{eee7}
\end{align}
Since we obtain
\begin{align*}
&\quad\eqref{eee1}_4-\eqref{eee3}_3=0,\\
&\quad\eqref{eee1}_5-\eqref{eee3}_4\\
&=-\hbar S^{(1)}_{n-1,-n}S^{(2)}_{-n,n-1}-\hbar S^{(1)}_{n,-n+1}S^{(2)}_{-n+1,n}\\
&=-\hbar S^{(1)}_{n-1,-n}S^{(2)}_{-n,n-1}+\hbar S^{(1)}_{n-1,-n}(S^{(2)}_{-n,n-1}-\hbar S^{(1)}_{-n+1,n})=-\hbar^2 S^{(1)}_{n-1,-n}S^{(1)}_{-n+1,n},\\
&\quad\eqref{eee1}_6-\eqref{eee5}_4\\
&=\hbar S^{(2)}_{n-1,-n}S^{(1)}_{-n+1,n}-\hbar S^{(2)}_{n,-n+1}S^{(1)}_{-n+1,n}=-\hbar^2 S^{(1)}_{n-1,-n}S^{(1)}_{-n+1,n},\\
&\quad-(n-1)\eqref{eee4}_4-(n-1)\eqref{eee6}_4+\eqref{eee7}_3=0,\\
&\quad\eqref{eee2}_1-(n-1)\eqref{eee4}_1\\
&\qquad\qquad-(n-1)\eqref{eee4}_2-\eqref{eee6}_1-(n-1)\eqref{eee6}_2-(n-1)\eqref{eee4}_3-(n-1)\eqref{eee6}_3\\
&=-2(n-1)\hbar \Phi(H_{n-1,1})+2(n-1)^2\hbar^2 \Phi(H_{n-1,0})
\end{align*}
by \eqref{SS5}, \eqref{SS6} and a direct computation, we obtain
\begin{align*}
&\quad[\Phi(X^+_{n-1,1}),\Phi(X^-_{n-1,1})]\\
&=S^{(3)}_{n-1,n-1}-S^{(3)}_{n,n}+\hbar S^{(1)}_{n,n}S^{(2)}_{n-1,n-1}+\hbar S^{(2)}_{n,n}S^{(1)}_{n-1,n-1}\\
&\quad-2(n-1)\hbar \Phi(H_{n-1,1})-(n-1)^2\hbar^2 \Phi(H_{n-1,0})\\
&\quad-\hbar \sum_{a\in I_{n-1}}(S^{(2)}_{n-1,a}S^{(1)}_{a,n-1}+S^{(1)}_{n-1,a}S^{(2)}_{a,n-1})+\hbar \sum_{a\in I_{n-1}}(S^{(1)}_{n,a}S^{(2)}_{a,n}+S^{(2)}_{n,a}S^{(1)}_{a,n})\\
&\quad-\hbar^2 \sum_{b\in I_{n-1}}S^{(1)}_{n-1,n-1}S^{(1)}_{n,b}S^{(1)}_{b,n}+\hbar^2 S^{(1)}_{n,-n+1}S^{(1)}_{-n,n-1}\\
&\quad+\hbar^2 \sum_{a,b\in I_{n-1}}S^{(1)}_{n-1,a}S^{(1)}_{a,b}S^{(1)}_{b,n-1}-\hbar^2 \sum_{a\in I_{n-1}}S^{(1)}_{n-1,a}S^{(1)}_{n,n}S^{(1)}_{b,n-1}-\hbar^2 \sum_{a,b\in I_{n-1}}S^{(1)}_{n,b}S^{(1)}_{b,a}S^{(1)}_{a,n}
\end{align*}
Next, we compute $[[\Phi(X^+_{n-1,1}),\Phi(X^-_{n-1,1})],X^+_{n,0}]$. By \eqref{SS1} and \eqref{SS6.5}, we obtain
\begin{align}
&\quad[S^{(3)}_{n-1,n-1}-S^{(3)}_{n,n},S^{(1)}_{n-1,-n}]=S^{(3)}_{n-1,-n}+S^{(3)}_{n,-n+1}=0,\label{qqq1}\\
&\quad[\hbar S^{(1)}_{n,n}S^{(2)}_{n-1,n-1}+\hbar S^{(2)}_{n,n}S^{(1)}_{n-1,n-1},S^{(1)}_{n-1,-n}]\nonumber\\
&=\hbar S^{(1)}_{n,n}S^{(2)}_{n-1,-n}-\hbar S^{(1)}_{n,-n+1}S^{(2)}_{n-1,n-1}-\hbar S^{(2)}_{n,-n+1}S^{(1)}_{n-1,n-1}+\hbar S^{(2)}_{n,n}S^{(1)}_{n-1,-n},\label{qqq2}\\
&\quad[\hbar \sum_{a\in I_{n-1}}S^{(2)}_{n-1,a}S^{(1)}_{a,n-1},S^{(1)}_{n-1,-n}]\nonumber\\
&=\hbar \sum_{a\in I_{n-1}}S^{(2)}_{n-1,a}S^{(1)}_{a,-n}+\hbar S^{(2)}_{n-1,-n+1}S^{(1)}_{n,n-1}+\hbar S^{(2)}_{n-1,-n}S^{(1)}_{n-1,n-1},\label{qqq3}\\
&\quad[\hbar \sum_{a\in I_{n-1}}S^{(1)}_{n-1,a}S^{(2)}_{a,n-1},S^{(1)}_{n-1,-n}]=\hbar \sum_{a\in I_{n-1}}S^{(1)}_{n-1,a}S^{(2)}_{a,-n}+\hbar S^{(1)}_{n-1,-n}S^{(2)}_{n-1,n-1},\label{qqq4}\\
&\quad[\hbar \sum_{a\in I_{n-1}}S^{(1)}_{n,a}S^{(2)}_{a,n},S^{(1)}_{n-1,-n}]=-\hbar \sum_{a\in I_{n-1}}S^{(1)}_{n,a}S^{(2)}_{a,-n+1}+\hbar S^{(1)}_{n,-n+1}S^{(2)}_{n,n},\label{qqq5}\\
&\quad[\hbar \sum_{a\in I_{n-1}}S^{(2)}_{n,a}S^{(1)}_{a,n},S^{(1)}_{n-1,-n}]=-\hbar \sum_{a\in I_{n-1}}S^{(2)}_{n,a}S^{(1)}_{a,-n+1}+\hbar S^{(2)}_{n,-n+1}S^{(1)}_{n,n}+\hbar S^{(2)}_{n,-n}S^{(1)}_{n-1,n},\label{qqq6}\\
&\quad [\hbar^2 \sum_{b\in I_{n-1}}S^{(1)}_{n-1,n-1}S^{(1)}_{n,b}S^{(1)}_{b,n},S^{(1)}_{n-1,-n}]\nonumber\\
&=-\hbar^2 \sum_{b\in I_{n-1}}S^{(1)}_{n-1,n-1}S^{(1)}_{n,b}S^{(1)}_{b,-n+1}+\hbar^2 S^{(1)}_{n-1,n-1}S^{(1)}_{n,-n+1}S^{(1)}_{n,n}+\hbar^2 \sum_{b\in I_{n-1}}S^{(1)}_{n-1,-n}S^{(1)}_{n,b}S^{(1)}_{b,n},\label{qqq7}\\
&\quad[\hbar^2 S^{(1)}_{n,-n+1}S^{(1)}_{-n,n-1},S^{(1)}_{n-1,-n}]=\hbar^2 S^{(1)}_{n,-n+1}[S^{(1)}_{-n,n-1},S^{(1)}_{n-1,-n}]\nonumber\\
&=\hbar^2 S^{(1)}_{n,-n+1}S^{(1)}_{-n,-n}-\hbar^2 S^{(1)}_{n,-n+1}S^{(1)}_{n-1,n-1},\label{qqq8}\\
&\quad[\hbar^2 \sum_{a,b\in I_{n-1}}S^{(1)}_{n-1,a}S^{(1)}_{a,b}S^{(1)}_{b,n-1},S^{(1)}_{n-1,-n}]\nonumber\\
&=\hbar^2 \sum_{a,b\in I_{n-1}}S^{(1)}_{n-1,a}S^{(1)}_{a,b}S^{(1)}_{b,-n}+\hbar^2 \sum_{a\in I_{n-1}}S^{(1)}_{n-1,a}S^{(1)}_{a,-n+1}S^{(1)}_{n,n-1}\nonumber\\
&\quad+\hbar^2 \sum_{a\in I_{n-1}}S^{(1)}_{n-1,a}S^{(1)}_{a,-n}S^{(1)}_{n-1,n-1}+\hbar^2 \sum_{b\in I_{n-1}}S^{(1)}_{n-1,-n}S^{(1)}_{n-1,b}S^{(1)}_{b,n-1},\label{qqq9}\\
&\quad[\hbar^2 \sum_{a\in I_{n-1}}S^{(1)}_{n-1,a}S^{(1)}_{n,n}S^{(1)}_{a,n-1},S^{(1)}_{n-1,-n}]\nonumber\\
&=\hbar^2 \sum_{a\in I_{n-1}}S^{(1)}_{n-1,a}S^{(1)}_{n,n}S^{(1)}_{a,-n}-\hbar^2 \sum_{a\in I_{n-1}}S^{(1)}_{n-1,a}S^{(1)}_{n,-n+1}S^{(1)}_{a,n-1}+\hbar^2 S^{(1)}_{n-1,-n}S^{(1)}_{n,n}S^{(1)}_{n-1,n-1},\label{qqq10}\\
&\quad[\hbar^2 \sum_{a,b\in I_{n-1}}S^{(1)}_{n,b}S^{(1)}_{b,a}S^{(1)}_{a,n},S^{(1)}_{n-1,-n}]\nonumber\\
&=-\hbar^2 \sum_{a,b\in I_{n-1}}S^{(1)}_{n,b}S^{(1)}_{b,a}S^{(1)}_{a,-n+1}+\hbar^2 \sum_{b\in I_{n-1}}S^{(1)}_{n,b}S^{(1)}_{b,-n+1}S^{(1)}_{n,n}\nonumber\\
&\quad+\hbar^2 \sum_{b\in I_{n-1}}S^{(1)}_{n,b}S^{(1)}_{b,-n}S^{(1)}_{n-1,n}+\hbar^2 \sum_{a\in I_{n-1}}S^{(1)}_{n,-n+1}S^{(1)}_{n,a}S^{(1)}_{a,n}.\label{qqq11}
\end{align}
By \eqref{SS5}, \eqref{SS6} and \eqref{SS1}, we have
\begin{align}
&\quad\eqref{qqq2}_3-\eqref{qqq3}_3=-2\hbar S^{(2)}_{n-1,-n}S^{(1)}_{n-1,n-1}+\hbar^2 S^{(1)}_{n,-n+1}S^{(1)}_{n-1,n-1},\label{rrr1}\\
&\quad-\eqref{qqq3}_1+\eqref{qqq5}_2=-(2n-2)\hbar S^{(2)}_{n,-n+1}-\hbar S^{(2)}_{n-1,-n}+\hbar^2 \sum_{a\in I_n}S^{(1)}_{n-1,a}(S^{(1)}_{a,-n}),\label{rrr2}\\
&\quad\eqref{qqq2}_4+\eqref{qqq5}_2=-\hbar S^{(2)}_{n,-n+1},\label{rrr3}\\
&\quad\eqref{qqq2}_2-\eqref{qqq4}_2=0,\label{rrr4}\\
&\quad-\eqref{qqq4}_1+\eqref{qqq5}_1=\hbar^2 \sum_{a\in I_{n-1}}S^{(1)}_{n-1,a}S^{(1)}_{a,-n}-(2n-2)\hbar S^{(2)}_{n,-n+1}+\hbar S^{(2)}_{n,-n+1},\label{rrr5}\\
&\quad\eqref{qqq2}_1+\eqref{qqq6}_2=2\hbar S^{(2)}_{n,-n+1}S^{(1)}_{n,n}+\hbar S^{(2)}_{n,-n+1}-\hbar^2 S^{(1)}_{n,n}S^{(1)}_{n-1,-n},\label{rrr6}\\
&\quad-\eqref{qqq7}_1+\eqref{qqq9}_3\nonumber\\
&=\hbar^2 \sum_{b\in I_{n-1}}S^{(1)}_{n,b}S^{(1)}_{b,-n+1}+2\hbar^2 \sum_{b\in I_{n-1}}S^{(1)}_{n-1,b}S^{(1)}_{b,-n}+(2n-3)\hbar^2 S^{(1)}_{n,-n+1}S^{(1)}_{n-1,n-1},\label{rrr7}\\
&\quad-\eqref{qqq7}_3-\eqref{qqq11}_4=0,\label{rrr8}\\
&\quad-\eqref{qqq7}_2-\eqref{qqq10}_3=\hbar^2 S^{(1)}_{n-1,-n}S^{(1)}_{n,n},\label{rrr9}\\
&\quad-\eqref{qqq10}_1-\eqref{qqq11}_2\nonumber\\
&=-(2n-1)\hbar^2 S^{(1)}_{n-1,-n}S^{(1)}_{n,n}-\hbar^2 \sum_{a\in I_{n-1}}S^{(1)}_{n-1,a}S^{(1)}_{a,-n}-2\hbar^2 \sum_{b\in I_{n-1}}S^{(1)}_{n,b}S^{(1)}_{b,-n+1}S^{(1)}_{n,n},\label{rrr10}\\
&\quad\eqref{qqq9}_1-\eqref{qqq11}_1=(2n-3)\hbar^2 \sum_{a\in I_{n-1}}S^{(1)}_{n-1,a}S^{(1)}_{a,-n}+(2n-3)\hbar^2 \sum_{b\in I_{n-1}}S^{(1)}_{b,-n}S^{(1)}_{n-1,b},\label{rrr11}\\
&\quad\eqref{qqq9}_4-\eqref{qqq10}_2=-\hbar^2 S^{(1)}_{n-1,-n}S^{(1)}_{n-1,n-1}.\label{rrr12}
\end{align}
Since we obtain
\begin{align*}
[\Phi(X^+_{n-1,1}),X^+_{n,0}]
&=S^{(2)}_{n-1,-n+1}-\hbar \sum_{a\in I_n}S^{(1)}_{n-1,a}S^{(1)}_{a,-n+1}+\hbar S^{(1)}_{n-1,-n}S^{(1)}_{n-1,n},\\
[\Phi(X^-_{n-1,1}),X^+_{n,0}]
&=-S^{(2)}_{n,-n}+\hbar \sum_{a\in I_n}S^{(1)}_{n,a}S^{(1)}_{a,-n}+\hbar S^{(1)}_{n,-n+1}S^{(1)}_{n,n-1},
\end{align*}
we have
\begin{align}
-\eqref{qqq3}_2+\eqref{qqq9}_2
&=-\hbar [\Phi(X^+_{n-1,1}),X^+_{n,0}]S^{(1)}_{n,n-1}+\hbar^2 S^{(1)}_{n-1,-n}S^{(1)}_{n-1,n}S^{(1)}_{n,n-1},\label{rrr13}\\
\eqref{qqq6}_3-\eqref{qqq11}_3
&=-\hbar [\Phi(X^-_{n-1,1}),X^+_{n,0}]S^{(1)}_{n-1,n}+\hbar^2 S^{(1)}_{n,-n+1}S^{(1)}_{n,n-1}S^{(1)}_{n-1,n}.\label{rrr14}
\end{align}
Since
\begin{align*}
&\quad\eqref{rrr1}_1+\eqref{rrr7}_2=2\hbar \Phi(X^+_{n,1})S^{(1)}_{n,n}+2(n-1)\hbar^2 S^{(1)}_{n-1,-n}S^{(1)}_{n,n},\\
&\quad\eqref{rrr6}_1+\eqref{rrr10}_3=-2\hbar \Phi(X^+_{n,1})S^{(1)}_{n-1,n-1}-2(n-1)\hbar^2 S^{(1)}_{n-1,-n}S^{(1)}_{n-1,n-1},\\
&\quad\eqref{rrr2}_3+\eqref{rrr5}_1+\eqref{rrr10}_2+\eqref{rrr11}_1=(2n-2)\hbar^2 \sum_{a\in I_{n-1}}S^{(1)}_{n-1,a}S^{(1)}_{a,-n},\\
&\quad\eqref{rrr7}_1+\eqref{rrr11}_2=(2n-2)\hbar^2 \sum_{a\in I_{n-1}}S^{(1)}_{n,a}S^{(1)}_{a,-n+1}\\
&=(2n-2)\hbar^2 \sum_{a\in I_{n-1}}S^{(1)}_{n-1,a}S^{(1)}_{a,-n}+(2n-2)\hbar^2 \sum_{a\in I_{n-1}}(S^{(1)}_{n,-n+1}-\delta_{a,-n+1}S^{(1)}_{n,-n+1})\\
&=(2n-2)\hbar^2 \sum_{a\in I_{n-1}}S^{(1)}_{n-1,a}S^{(1)}_{a,-n}+(2n-2)(2n-3)\hbar^2 S^{(1)}_{n,-n+1},\\
&\quad\eqref{rrr2}_1+\eqref{rrr2}_2+\eqref{rrr3}+\eqref{rrr5}_2+\eqref{rrr5}_3+\eqref{rrr6}_2\\
&=-\hbar S^{(2)}_{n-1,-n}-(4n-5)\hbar S^{(2)}_{n,-n+1}=-(4n-4)\hbar S^{(2)}_{n-1,-n}-(4n-5)\hbar^2S^{(1)}_{n,-n+1}\\
&\quad\eqref{qqq8}_1+\eqref{rrr6}_3+\eqref{rrr9}+\eqref{rrr10}_1\\
&=-\hbar^2[S^{(1)}_{n-1,-n},S^{(1)}_{n,n}]-2(n-1)\hbar^2S^{(1)}_{n-1,-n}S^{(1)}_{n,n}=\hbar^2 S^{(1)}_{n-1,-n}-2(n-1)\hbar^2 S^{(1)}_{n-1,-n}S^{(1)}_{n,n},\\
&\quad\eqref{qqq8}_2+\eqref{rrr1}_2+\eqref{rrr7}_3+\eqref{rrr12}=(2n-2)\hbar^2 S^{(1)}_{n,-n+1}S^{(1)}_{n-1,n-1},
\end{align*}
we have
\begin{align*}
&\quad[[\Phi(X^+_{n-1,1}),\Phi(X^-_{n-1,1})],X^+_{n,0}]\\
&=-2(n-1)\hbar [\Phi(H_{n-1,1}),X^+_{n,0}]-2\hbar \Phi(X^+_{n,1})\Phi(H_{n-1,0})\\
&\quad-2(n-1)\hbar \Phi(X^+_{n,0})(\Phi(H_{n-1,0})+\Phi(H_{n,0}))\\
&\quad+(4n-4)\hbar \Phi(X^+_{n,1})+6(n-1)\hbar^2 \Phi(X^+_{n,0}).
\end{align*}
\subsection{Compatibility with \eqref{q903-1}}
By \eqref{SS1}, we obtain
\begin{align*}
&\quad[\Phi(X^+_{n,1}),[\Phi(X^+_{n-2,0}),\Phi(X^+_{n-1,0})]]\\
&=[-S^{(2)}_{n-1,-n}-(n-1)\hbar S^{(1)}_{n-1,-n}+\hbar\sum_{a\in I_{n-1}}S^{(1)}_{n-1,a}S^{(1)}_{a,-n},[S^{(1)}_{n-2,n-1},S^{(1)}_{n-1,n}]]\\
&=[-S^{(2)}_{n-1,-n}-(n-1)\hbar S^{(1)}_{n-1,-n}+\hbar\sum_{a\in I_{n-1}}S^{(1)}_{n-1,a}S^{(1)}_{a,-n},S^{(1)}_{n-2,n}]\\
&=S^{(2)}_{n-1,-n+2}+(n-1)\hbar S^{(1)}_{n-1,-n+2}\\
&\quad-\hbar\sum_{a\in I_{n-1}}S^{(1)}_{n-1,a}S^{(1)}_{a,-n+2}+\hbar S^{(1)}_{n-1,n}S^{(1)}_{n-2,-n}+\hbar S^{(1)}_{n-1,-n+2}S^{(1)}_{-n,-n}
\end{align*}
and
\begin{align*}
&\quad[[\Phi(X^+_{n,1}),\Phi(X^+_{n-1,0})],\Phi(X^+_{n-2,0})]\\
&=[[-S^{(2)}_{n-1,-n}-(n-1)\hbar S^{(1)}_{n-1,-n}+\hbar\sum_{a\in I_{n-1}}S^{(1)}_{n-1,a}S^{(1)}_{a,-n},S^{(1)}_{n-2,n-1}],S^{(1)}_{n-1,n}]]\\
&=[S^{(2)}_{n-2,-n}+(n-1)\hbar S^{(1)}_{n-2,-n}\\
&\quad-\hbar\sum_{a\in I_{n-1}}S^{(1)}_{n-2,a}S^{(1)}_{a,-n},S^{(1)}_{n-1,n}]\\
&=-S^{(2)}_{n-2,-n+1}-(n-1)\hbar S^{(1)}_{n-2,-n+1}\\
&\quad+\hbar\sum_{a\in I_{n-1}}S^{(1)}_{n-2,a}S^{(1)}_{a,-n+1}-\hbar S^{(1)}_{n-2,n}S^{(1)}_{n-1,-n}-\hbar S^{(1)}_{n-2,-n+1}S^{(1)}_{-n,-n}
\end{align*}
Thus, we obtain
\begin{align*}
&\quad[\Phi(X^+_{n,1}),[\Phi(X^+_{n-2,0}),\Phi(X^+_{n-1,0})]]\\
&=S^{(2)}_{n-1,-n+2}+S^{(2)}_{n-2,-n+1}-\hbar\sum_{a\in I_{n-1}}S^{(1)}_{n-1,a}S^{(1)}_{a,-n+2}-\hbar\sum_{a\in I_{n-1}}S^{(1)}_{n-2,a}S^{(1)}_{a,-n+1}\\
&\quad+\hbar S^{(1)}_{n-1,n}S^{(1)}_{n-2,-n}+\hbar S^{(1)}_{n-2,n}S^{(1)}_{n-1,-n}\\
&=-\hbar S^{(1)}_{n-1,-n+2}-(n-1)S^{(1)}_{n-1,-n+2}+\hbar S^{(1)}_{n-1,n}S^{(1)}_{n-2,-n}+\hbar S^{(1)}_{n-2,n}S^{(1)}_{n-1,-n}\\
&=-(n-1)S^{(1)}_{n-1,-n+2}+\hbar S^{(1)}_{n-1,n}S^{(1)}_{n-2,-n}+\hbar S^{(1)}_{n-1,-n}S^{(1)}_{n-2,n}.
\end{align*}
\subsection{Compatibility with \eqref{eq903-3}}
By \eqref{SS1}, we obtain
\begin{align*}
&\quad[[\Phi(X^+_{n,0}),\Phi(X^+_{n-2,0})],[\Phi(X^+_{n,0}),\Phi(X^+_{n-1,1})]]\\
&=-[S^{(1)}_{n-2,-n},[S^{(1)}_{n-1,-n},-S^{(2)}_{n-1,n}-(n-1)\hbar S^{(1)}_{n-1,n}+\hbar\sum_{a\in I_{n-1}}S^{(1)}_{n-1,a}S^{(1)}_{a,n}]]\\
&=-[S^{(1)}_{n-2,-n},-\hbar\sum_{a\in I_{n-1}}S^{(1)}_{n-1,a}S^{(1)}_{a,-n+1}-\hbar S^{(1)}_{n-1,-n}S^{(1)}_{n-1,n}]\\
&=\hbar\sum_{a\in I_{n-1}}[S^{(1)}_{n-2,-n},S^{(1)}_{n-1,a}S^{(1)}_{a,-n+1}]+\hbar S^{(1)}_{n-1,-n}[S^{(1)}_{n-2,-n},S^{(1)}_{n-1,n}]\\
&=-\hbar S^{(1)}_{n-1,-n}S^{(1)}_{n-2,-n+1}-\hbar S^{(1)}_{n-1,-n+2}S^{(1)}_{n,-n+1}+\hbar S^{(1)}_{n-1,-n}S^{(1)}_{n-1,-n+2}\\
&=-\hbar S^{(1)}_{n-1,-n}S^{(1)}_{n-2,-n+1}.
\end{align*}

\subsection{Compatibility with \eqref{eq903-7}}
By \eqref{SS1}, we obtain
\begin{align*}
&\quad[[\Phi(X^+_{n-1,0}),\Phi(X^+_{n-2,0})],[\Phi(X^+_{n-1,0}),\Phi(X^+_{n,1})]]\\
&=-[S^{(1)}_{n-2,n},[S^{(1)}_{n-1,n},-S^{(2)}_{n-1,-n}-(n-1)\hbar S^{(1)}_{n-1,-n}+\hbar\sum_{a\in I_{n-1}}S^{(1)}_{n-1,a}S^{(1)}_{a,-n}]]\\
&=-[S^{(1)}_{n-2,n},\hbar\sum_{a\in I_{n-1}}S^{(1)}_{n-1,a}S^{(1)}_{a,-n+1}-\hbar S^{(1)}_{n-1,n}S^{(1)}_{n-1,-n}]\\
&=\hbar S^{(1)}_{n-1,n}S^{(1)}_{n-2,-n+1}+\hbar S^{(1)}_{n-1,-n+2}S^{(1)}_{-n,-n+1}+\hbar S^{(1)}_{n-1,n}S^{(1)}_{n-1,-n+2}\\
&=\hbar S^{(1)}_{n-2,-n+1}S^{(1)}_{n-1,n}.
\end{align*}

\section*{Data Availability}
The authors confirm that the data supporting the findings of this study are available within the article and its supplementary materials.
\section*{Declarations}
\subsection*{Funding}
This work was supported by JSPS Overseas Research Fellowships, Grant Number JP2360303. 
\subsection*{Conflicts of interests/Competing interests}
The authors have no competing interests to declare that are relevant to the content of this article.

\end{document}